\documentclass[11pt]{amsart}

\usepackage{amsmath,amscd,amssymb,amsthm,mathtools,enumerate,subcaption}
\usepackage{mwe}

\usepackage{xcolor}
\usepackage[pdfencoding=auto]{hyperref} 

\hypersetup{
     colorlinks,
     linkcolor= {blue!80!black},
     citecolor={blue!80!black},
     urlcolor={blue!80!black}}
     
\usepackage[nameinlink,capitalise]{cleveref}

\usepackage[a4paper,left=3cm,right=3cm,top=4.5cm,bottom=4.5cm]{geometry}

\usepackage{parskip}
\usepackage{tikz-cd}

\newcommand{\nwc}{\newcommand}

\nwc{\mf}{\mathbf} 
\nwc{\blds}{\boldsymbol} 
\nwc{\ml}{\mathcal} 

\newcommand{\mB}{\mathrm{B}} 
\newcommand{\C}{\mathbb{C}} 
\newcommand{\F}{\mathcal{F}} 
\newcommand{\gH}{\mathfrak{H}} 
\renewcommand{\H}{\mathfrak{H}} 
\newcommand{\N}{\mathbb{N}}
\newcommand{\R}{\mathbb{R}}
\newcommand{\Sp}{\mathbb{S}} 

\renewcommand{\epsilon}{\varepsilon}

\renewcommand{\d}[1]{\mathrm{d}#1}
\newcommand{\abs}[1]{\left\lvert #1 \right\rvert}
\newcommand{\norm}[1]{\left\lVert #1 \right\rVert}

\newcommand{\normm}[1]{{\left\vert\kern-0.25ex\left\vert\kern-0.25ex\left\vert #1 
    \right\vert\kern-0.25ex\right\vert\kern-0.25ex\right\vert}}

\newcommand{\cc}[1]{\overline{#1}}
\newcommand{\hp}[2]{\left\langle #1 ,#2\right\rangle}

\newcommand{\br}[1]{\left \langle #1 \right \rangle} 
\newcommand{\ol}[1]{\overline{#1}}

\newcommand{\p}{\partial}
\newcommand{\To}{\longrightarrow}

\nwc{\Spec}{\operatorname{Spec}}

\newcommand{\loc}{\mathrm{loc}}
\newcommand{\rad}{\mathrm{rad}}

\let\Re\relax
\DeclareMathOperator{\Re}{\mathrm{Re}}
\let\Im\relax
\DeclareMathOperator{\Im}{\mathrm{Im}}
\DeclareMathOperator{\Id}{\mathrm{Id}}

\DeclareMathOperator{\SO}{SO}

\DeclareMathOperator{\supp}{supp}

\newcommand{\ka}{\kappa}

\newcommand{\Tr}{\mathrm{Tr}}
\newcommand{\qs}{q^s_\ka}
\newcommand{\qb}{q^\mB_\ka}

\nwc{\IA}{\mathbb{A}} 
\nwc{\IB}{\mathbb{B}} 
\nwc{\IC}{\mathbb{C}} 
\nwc{\ID}{\mathbb{D}} 
\nwc{\IE}{\mathbb{E}} 
\nwc{\IF}{\mathbb{F}} 
\nwc{\IG}{\mathbb{G}} 
\nwc{\IH}{\mathbb{H}} 
\nwc{\IN}{\mathbb{N}} 
\nwc{\IP}{\mathbb{P}} 
\nwc{\IQ}{\mathbb{Q}} 
\nwc{\IR}{\mathbb{R}} 
\nwc{\IS}{\mathbb{S}} 
\nwc{\IT}{\mathbb{T}} 
\nwc{\IZ}{\mathbb{Z}} 

\nwc{\cA}{\ml{A}}
\nwc{\cB}{\ml{B}}
\nwc{\cC}{\ml{C}}
\nwc{\cD}{\ml{D}}
\nwc{\cE}{\ml{E}}
\nwc{\cF}{\ml{F}}
\nwc{\cG}{\ml{G}}
\nwc{\cH}{\ml{H}}
\nwc{\cI}{\ml{I}}
\nwc{\cJ}{\ml{J}}
\nwc{\cK}{\ml{K}}
\nwc{\cL}{\ml{L}}
\nwc{\cM}{\ml{M}}
\nwc{\cN}{\ml{N}}
\nwc{\cO}{\ml{O}}
\nwc{\cP}{\ml{P}}
\nwc{\cQ}{\ml{Q}}
\nwc{\cR}{\ml{R}}
\nwc{\cS}{\ml{S}}
\nwc{\cT}{\ml{T}}
\nwc{\cU}{\ml{U}}
\nwc{\cV}{\ml{V}}
\nwc{\cW}{\ml{W}}
\nwc{\cX}{\ml{X}}
\nwc{\cY}{\ml{Y}}
\nwc{\cZ}{\ml{Z}}

\newtheorem{main-theorem}{Theorem}
\newtheorem{proposition}{Proposition}[section]

\newtheorem{lemma}[proposition]{Lemma}
\newtheorem{theorem}[proposition]{Theorem}

\theoremstyle{remark}
\newtheorem{remark}[proposition]{Remark}

\numberwithin{equation}{section}

\title{The Born approximation for the fixed energy Calderón problem}

\author[F. Macià]{Fabricio Macià}
\address[FM]{M$^2$ASAI. Universidad Politécnica de Madrid, ETSI Navales, Avda. de la Memoria, 4, 28040, Madrid, Spain.}
\email{fabricio.macia@upm.es}

\author[C. Meroño]{Cristóbal Meroño}
\address[CM]{M$^2$ASAI. Universidad Politécnica de Madrid, ETSI Navales, Avda. de la Memoria, 4, 28040, Madrid, Spain.}
\email{cj.merono@upm.es}

\author[D. Sánchez-Mendoza]{Daniel Sánchez-Mendoza}
\address[DSM]{M$^2$ASAI. Universidad Politécnica de Madrid, ETSI Navales, Avda. de la Memoria, 4, 28040, Madrid, Spain.}
\email{daniel.sanchezmen@upm.es}

\begin{document}

\begin{abstract}
The Born approximation of a potential in the context of the Calderón inverse problem is an object that can be formally defined in terms of spectral data of the Dirichlet-to-Neumann map of the corresponding Schrödinger operator.
In this article, we prove, in the case of radial potentials in the Euclidean ball and any fixed energy, that the Born approximation is well-defined as a compactly supported radial distribution, and that the Calderón problem can be reformulated as recovering a potential from its Born approximation. In addition,  we show that the Born approximation depends locally on the potential and captures exactly its singularities, and that the functional that maps the Born approximation to the potential is Hölder continuous. We also prove that the Born approximation converges to the potential in the high-energy limit. Moreover, we give an explicit formula for the Fourier transform of the Born approximation at any fixed energy, and illustrate how it can be used as the basis of an accurate procedure to approximate a potential from its Dirichlet-to-Neumann map.

\end{abstract}

\maketitle

\section{Introduction}

\subsection{Outline of the article}
The Calderón problem for a Schrödinger operator, also known as the Gel'fand-Calderón problem, asks for the reconstruction of a potential from the knowledge of certain boundary measurements of the solutions to the corresponding Schrödinger equation; see \cite{calderon,calderon_rep} and \cite{Gelfand_54}.
In this article, we are interested in the fixed-energy version of this problem: given an open, bounded domain $\Omega\subset\IR^d$, $d\geq2$, with $\cC^1$ boundary, an energy $\ka\in\IC$ and a real-valued potential $q$, consider the elliptic boundary-value problem for the Helmholtz equation
\begin{equation} \label{eq:boundary_problem}
    \left\{
\begin{array}{rll}
-\Delta u(x) + q(x) u(x) -\ka u(x) &=0,    \quad   x\in\Omega,\\
u|_{\partial\Omega}  &= f.
\end{array}\right.
\end{equation}
The Dirichlet-to-Neumann (DtN) map associates the boundary value $f$ with the normal derivative on $\partial\Omega$ of the corresponding solution $u$ to \eqref{eq:boundary_problem}. Whenever $\ka$ is not a Dirichlet eigenvalue of $-\Delta+q$, this produces a well-defined linear operator
\begin{equation}\label{e:kqdtn}
    \begin{array}{cccc}
       \Lambda_{q,\kappa}:  & \cC^\infty(\partial\Omega) &\To&\cC^\infty(\partial\Omega) \\
         & f &\longmapsto & \partial_\nu u|_{\partial\Omega},
    \end{array}
\end{equation}
where $\nu$ is the vector field of exterior unit vectors normal to $\partial\Omega$.

The Gel'fand-Calderón problem, in its simplest form, consists in reconstructing $q$ from the knowledge of $\Lambda_{q,\kappa}$ for some fixed $\ka\in\IC$. Since it is known that $\Lambda_{q,\kappa}-\Lambda_{0,\kappa}$ is always an operator that is bounded in $L^2(\partial\Omega)$ (see for instance \Cref{sec:dtn}), it is convenient to encode the inverse problem using the nonlinear map
\begin{equation}\label{eq:defphik}
     \begin{array}{cccc}
       \Phi^\ka:  & \cX &\To& \cL^2(L^2(\partial\Omega)) \\
         & q &\longmapsto & \Lambda_{q,\kappa}-\Lambda_{0,\kappa},
    \end{array}    
\end{equation}
where $\cX$ is a class of admissible potentials for which the DtN map is defined. In this article $\cX$ will be a subset of $L^p(\Omega)$ for $p>1$ and $p \ge d/2$. The main issues one is interested in are:
\begin{enumerate}[i)]
    \item \textbf{Uniqueness. }Is every $q\in \cX$ uniquely determined by $\Lambda_{q,\ka}$? This amounts to showing that $\Phi^\ka$ is injective.
    \item \textbf{Stability. }Is the reconstruction process stable? That is, find a modulus of continuity for $\Phi^\ka$. This is not possible in general, since $(\Phi^\ka)^{-1}$ is not continuous (see, for example, \cite{Alessandrini_88,Ale_Cabib_08,Faraco_Kurylev_Ruiz_14}). 
    This is, therefore, an \textit{ill-posed} inverse problem. Nevertheless, one can ask for \textit{conditional stability} when additional requirements of regularity and boundedness are imposed on the class of admissible potentials.\footnote{Put in a more abstract setting, for a given \textit{compact} set $K\subset \cX$ one tries to compute the modulus of continuity of $(\Phi^\ka)^{-1}|_{\Phi^\ka(K)}$, which exists since $\Phi^\ka|_K$ is a uniformly continuous homeomorphism.}
    \item \textbf{Reconstruction. }Find an effective procedure to reconstruct $q$ from $\Lambda_{q,\ka}$, in other words, compute $(\Phi^\ka)^{-1}$. This is related to the problem of \textit{characterization} of the range $\Phi^\ka(\cX)$. 
\end{enumerate}

The uniqueness for $d \ge  3$ was proved in \cite{SU87}, and later in \cite{Na88,Novikov_1988} with particular emphasis on the case of fixed energy (see also \cite{CaroRogers16, Haberman_15_unbounded}). The two-dimensional case was solved in  ~\cite{Nachman_uniqueness_2d_96,Astala_Paivarinta_2006} for conductivities and \cite{Bukhgeim_2008,Blasten_2015} for potentials. These results are proved using the notion of Complex Geometric Optics solutions (CGO) from \cite{Fadd65} ($d\geq 3$) or different families of exponentially growing solutions of the equations when $d=2$.

Conditional stability was proved in dimensions $d\ge 3$ by \cite{Alessandrini_88} with a logarithmic-type modulus on continuity, which was shown to be optimal in \cite{Mand00} (see also \cite{KRS21} for a more detailed account of this issue). In dimension 2 it was proved in \cite{Barcelo_Barcelo_Ruiz_01, Clop_Faraco_Ruiz_10} (see also \cite{Far_Pr_18}). Improved stability at fixed energy was established in \cite{novikov_2011-stability,isakov_2011_stability, ILX20} in dimension three and \cite{santacesaria_2013_stability,santacesaria_2015_stability} in two dimensions. The instability estimates of \cite{Mand00} have been improved in the $\ka \neq 0$ case in \cite{Isa13, KUW21}.

Reconstruction is a difficult issue in general, both from analytical and numerical points of view. The classical approach to uniqueness, based on exponentially growing families of solutions, leads to reconstruction strategies that involve analyzing a certain scattering transform (see, for instance, \cite{Na88, Nachman_uniqueness_2d_96,Astala_Faraco_Rogers_2016}), that can be transformed into numerical algorithms (see, among many others, \cite{IMS00,KLMS2007,DHK12}). 
A different type of strategy, sometimes combined with the previous one, is based on linearization. This serves as the basis of one-step linearization methods (see \cite{HaSe2010}), the so-called Calderón method (see, for example, \cite{Bik_Mueller_08,  Mueller_second_order}), the algorithm described in \cite{Bikowski_Knudsen_Muller_11}, and has also been applied in deep learning approaches to the Calderón problem; see \cite{Martin2017}. 
This kind of linearization methods  are successful from the numerical point of view, but seem hard to justify rigorously. The main difficulty lies in proving the existence of a certain Born-type approximation for the inverse problem.

In this work, we address this question, showing the existence of a Born approximation for the inverse problem for radially symmetric potentials, and we analyze how this leads to interesting  uniqueness, stability, and characterization results. This approach does not use any CGO-type construction.

This approach was initiated  in \cite{BCMM_2022_Born, BCMM_2024_numerical} in the context of the Calderón problem, and is based on the notion of approximation in inverse scattering that can be traced back to the work of Born \cite{Born1926}. 
It was successfully applied in \cite{Radial_Born} to the Gel'fand-Calderón problem at zero energy; and a systematic exposition of this approach, which in principle applies to any inverse problem, was given in \cite{MaMe24}. In the present context, it can be described as follows. The map $\Phi^\ka$ is Fréchet differentiable and satisfies $\Phi^\ka(0)=0$; if for every $q\in \cX$ it is possible to find $\qb\in \cB$ in some space $\cB$ of functions or distributions solving 
\begin{equation}\label{eq:absba}
    \d\Phi^\ka_0(\qb) =  \Phi^\ka(q)= \Lambda_{q,\kappa}-\Lambda_{0,\kappa},
\end{equation}
then one has in fact obtained a factorization of $\Phi^\ka$ into a linear and a nonlinear map:  
\begin{equation}  \label{eq:diagram_1}
\begin{tikzcd}
\cX \arrow[dr,"\Phi^\ka_\mB"] \arrow[rr,"\Phi^{\ka}"] && \Phi^\ka(\cX) \\
 &  \cB \arrow[ur,"\d\Phi^\ka_0"]
\end{tikzcd}
\qquad \qquad
\begin{tikzcd}
q \; \ar[dr,|->,"\Phi_\mB^{\ka}"] \ar[rr,|->,"\Phi^{\ka}"] && \Lambda_{q,\ka}-\Lambda_{0,\ka} \\
 &  \qb \ar[ur,|->,"\d\Phi_0^{\ka}"]
\end{tikzcd}
\end{equation}
The objective is to exploit this factorization to obtain uniqueness, stability, reconstruction and characterization results for the inverse problem.
The main difficulty in implementing this strategy is to show that there exists a solution of equation \eqref{eq:absba}. In this work we show that \eqref{eq:absba} can be solved and that the factorization \eqref{eq:diagram_1} exists for the Gel'fand-Calderón problem at fixed energy $\ka\in\R$, assuming radial symmetry on the potentials. 

In the radial case, the existence of $\qb$ for $\ka=0$ has already been proved in \cite{Radial_Born}, a result that also implies a partial characterization of the set of DtN maps.  A key step in the proof involves the tools introduced by Simon \cite{Simon_spectral_I} in the context of inverse spectral theory of Schrödinger operators on the half-line, especially the notion of $A$-amplitude. In \cite{MaMe24}, the authors show that, in fact, the $A$-amplitude coincides with the notion of Born approximation for this one-dimensional inverse spectral problem.

Here, the proof of the existence of $\qb$ relies again on the analysis of a one-dimensional inverse spectral problem, the difficulty being that the corresponding Born approximation is no longer Simon's $A$-amplitude. The existence of such a Born approximation, which we call the $A_\ka$-amplitude, is an important part of this work. The main results of this article are:
\begin{enumerate}[i)]
    \item \textbf{Explicit description of }$\d\Phi^\ka_0$ \textbf{and its inverse. }We show that $\d\Phi^\ka_0$ maps potentials to operators that are rotationally invariant, and whose eigenvalues are certain moments of the potential. Also, we give an explicit formula for $(\d\Phi^\ka_0)^{-1}$. 
    This is presented in \Cref{main-thm:Fourier Radial} and \Cref{prop:frechet_der}.
    \item \textbf{Existence of the Born approximation. } This is stated in \Cref{main-thm:existence_Born}. 
    Using the explicit formula for $(\d\Phi^\ka_0)^{-1}$ and tools of spectral theory of Schrödinger operators on the half-line, most particularly the existence of an $A_\ka$-amplitude (\Cref{thm:A-amplitude}), we prove that $\qb$, the Born approximation at energy $\ka$, exists as a radial distribution. 
    We also present explicit formulas for $\qb$.
    \item \textbf{Regularity of $\qb$.} 
    In \Cref{main-thm:basic_properties} we prove that $\qb$ coincides with an integrable function outside the origin, that $\qb-q$ is one derivative smoother than $q$, and that $\qb =q$ at the boundary.
    \item \textbf{Injectivity of $\Phi_\mB^\ka$ and stability of $(\Phi_\mB^\ka)^{-1}$.}
    The map $\Phi_\mB^\ka: q \mapsto \qb$ is injective in a strong sense: two potentials coincide in a neighborhood of the boundary if and only if their Born approximations at energy $\ka$ coincide in that same neighborhood. In addition,  $(\Phi_\mB^\ka)^{-1}$ is Hölder continuous, under mild \textit{a priori} conditions on the potentials, with respect to the $L^1$ norm in the complement of any arbitrarily small ball centered in the origin. This is the scope of \Cref{main-thm:stability}, which is derived from its one-dimensional analogue \Cref{thm:stability_A_function}.
    \item \textbf{High-energy/semiclassical limit. }In \Cref{main-thm:limit} we show that the Born approximation at energy $\ka$ converges to the potential $q$ in the high-energy limit $\ka\to-\infty$. The analogous result for the direct problem, namely that $\Phi^\ka(q)-\d \Phi^\ka_0(q)$ converges to zero in the same regime, is proved in \Cref{prop:spectrum}.
\end{enumerate}
In \eqref{eq:diagram_1}, the maps $\d\Phi_0^{\ka}$ and $\Phi^\ka$ are continuous, but the ranges of each of them are not closed. Hence $(\d\Phi_0^\ka)^{-1}$ and $\Phi^\ka$ are both discontinuous and only conditionally stable. Therefore, an important consequence of iv)  is that  the bad behavior of the inverse of $\Phi^\ka$ is captured by the linear approximation $\d\Phi^\ka_0$, since the non-linear map $(\Phi_\mB^{\ka})^{-1}$ is Hölder stable.


\subsection{Statement of the main results}

From now on, we denote by $L^{p}_\rad(\IB^d,\R)$ the closed subspace of radial functions that belong to $L^p(\IB^d,\R)$. We will assume that $d\ge 2$ and that
\begin{equation}\label{eq:p_cond}
     p \in \R \text{ is admissible }\quad \iff \quad 1< p \le \infty ,\;  \text{and }\;  p \ge d/2.
\end{equation}
Also, we define
\begin{equation*}
    \quad \nu_d := \frac{d-2}{2},\qquad \IN_0:=\IN\cup\{0\}.  
\end{equation*}
Recall that the Dirichlet spectrum of $-\Delta$ on the ball $\IB^d$ can be expressed in terms of the zeros of certain Bessel functions. 
The operator $\Lambda_{q,\ka}$ is well defined as soon as $q\in L^{p}_\rad(\IB^d,\R)$, with $p$ admissible, and $\ka\in\IC\setminus\Spec_{H^1_0}(-\Delta+q)$. Since $q$ is radial, the DtN operator commutes with the action of $\SO(d)$ on $\IS^{d-1}$; as a consequence, its eigenfunctions are spherical harmonics (see \Cref{app:sh} for the definition and properties of spherical harmonics). More precisely, if $\gH^d_\ell$ denotes the subspace of spherical harmonics of degree $\ell\in\IN_0$ on $\IS^{d-1}$, then
\begin{equation} \label{id:DtN_map_eigenvalue}
    \Lambda_{q,\ka}|_{\gH^d_\ell}= \lambda_\ell[q,\kappa] \Id_{\gH^d_\ell}.
\end{equation} 
When $q=0$, the spectrum of $\Lambda_{0,\ka}$ can be explicitly computed. If $\ka =0$ 
then $\lambda_\ell[0,0] = \ell$ for every $\ell\in\N_0$, and, if  $\ka\in\IC\setminus\Spec_{H^1_0}(-\Delta)$ with $\ka\neq 0$, then
\begin{equation} \label{eq:free_dtn}
    \lambda_\ell[0,\kappa]=
        \ell - \sqrt{\kappa}\dfrac{J_{\ell+1+\nu_d}(\sqrt{\kappa})}{J_{\ell+\nu_d}(\sqrt{\kappa})},  \qquad \forall \ell \in \IN_0,
\end{equation}
where $J_\nu$ is the Bessel function of the first kind of index $\nu\in\IC$ (see \Cref{lemma:m_function_ka}). 

Given $\ell\in\IN_0$, the holomorphic function
\begin{equation}\label{eq:defvphi}
    \varphi_\ell(z):=\frac{J_{\ell+\nu_d}(z)}{z^{\nu_d}}=\frac{z^\ell}{2^{\ell+\nu_d}}\sum_{n=0}^\infty \frac{(-1)^n}{n!\Gamma(\ell+\nu_d+n+1)}\left(\frac{z}{2}\right)^{2n},\qquad z\in\IC,
\end{equation}
will play an important role throughout this work. For $\ka\in\IC$, the $\ka$-\textit{moments} of $q$ are:
\begin{equation}\label{e:moment}
    \sigma_\ell[q,\ka] :=  \frac{1}{|\IS^{d-1}|} \int_{\IB^d} q(x)  \varphi_\ell(\sqrt{\ka}|x|)^2  \, \d x,  \qquad \forall \ell \in \IN_0.
\end{equation}
Note that the functions $\ka\mapsto\sigma_\ell[q,\ka]$ are holomorphic in $\IC$.

The $\ka$-moments appear in the expression of the Fréchet differential of $\Phi^\ka$ at $q=0$. In \Cref{prop:frechet_der} we prove that, as soon as $\ka$ is not a Dirichlet eigenvalue of $-\Delta$ in the ball, $\d\Phi^\ka_0(q)$ is a bounded operator on $L^2(\IS^{d-1})$ for every $q\in L^{\infty}_\rad(\IB^d,\R)$. These operators turn out to be invariant by the action of $\SO(d)$ and are characterized by their restriction on spherical harmonics, which is precisely given by
\begin{equation}\label{eq:difac}
    \varphi_\ell(\sqrt{\ka})^2\d\Phi^\ka_0 (q)|_{\gH^d_\ell} = \sigma_\ell[q,\ka] \Id_{\gH^d_\ell},   \qquad \forall \ell \in \IN_0.
\end{equation}
Therefore, by \eqref{id:DtN_map_eigenvalue}, that a function $\qb$ satisfies $\d\Phi^\ka_0 (\qb)=\Lambda_{q,\ka}-\Lambda_{0,\ka}$ is formally equivalent to the fact that $\qb$ solves the following moment problem:
\begin{equation} \label{eq:moment_prob_formal}
   \sigma_\ell[\qb,\ka] = (\lambda_\ell[q,\kappa]-\lambda_\ell[0,\kappa])\varphi_\ell(\sqrt{\ka})^2, \qquad \qquad \forall \ell \in \IN_0.
\end{equation}
Note right away that, by \eqref{eq:free_dtn} and \eqref{eq:defvphi}, the left-hand side of the identity \eqref{eq:moment_prob_formal} is well defined as soon as $\ka\in\IC\setminus \Spec_{H^1_0}(-\Delta+q)$ even if $ \ka\in \Spec_{H^1_0}(-\Delta)$.

The existence of a solution for \eqref{eq:moment_prob_formal} is not guaranteed. In order to prove that a solution exists, we need to allow for the possibility of $\qb$ being a distribution. Let $U\subseteq \R^d$ be an open set. We denote by $\cE'(U)$  the space of distributions with compact support in $U$, and $\cE'_\rad(U) \subset \cE'(U)$ the subset of radially symmetric distributions. 

The definition \eqref{e:moment} can be extended from $L^{\infty}_\rad(\IB^d,\R)$ to the space $\cE_\rad'(\R^d)$ as follows: 
for $f\in \cE_\rad'(\R^d)$, $\ka\in\IC$ and $\ell\in\IN_0$ define
\begin{equation} \label{e:moment_distributions}
    \sigma_\ell[f,\kappa]:= \frac{1}{\abs{\Sp^{d-1}}}\br{f, (\varphi_{\ka,\ell})^2}_{\cE'\times\cC^\infty},\quad \text{ where }\varphi_{\ka,\ell}(x):=\varphi_\ell(\sqrt{\ka}|x|),
\end{equation}
where $\br{\cdot, \cdot}_{\cE'\times\cC^\infty}$ denotes the duality pairing in $\cE'(\R^d)$. This makes sense since $(\varphi_{\ka,\ell})^2\in\cC^\infty(\IR^d)$.

To state our first result, we define
\begin{equation*}
    \cB_d := L^1_\rad(\IB^d) + \cE'_\rad(\IB^d).
\end{equation*}
The elements of $\cB_d$ are radial distributions in $\IB^d$ that coincide with $L^1$ functions in a neighborhood of the boundary of $\IB^d$.
This implies that the extension by zero of a distribution in $\cB_d $ is an element of $\cE'_\rad(\R^d)$. In this way, when convenient, $\cB_d$ can be identified with a subset of $\cE'_\rad(\R^d)$. We adopt this perspective in a few points in this work, for example, to define the moments and the Fourier transform of elements of $\cB_d$ using \eqref{e:moment_distributions} and \eqref{eq:Fourier_distribution}. Note that \eqref{eq:difac} allows us to extend $\d\Phi^\ka_0$ to $\cB_d$.

Let
    \begin{equation*}
        X_{p,\ka}(\IB^d):=\{q\in L^p_\rad(\IB^d,\IR)\,:\,\ka\not\in\Spec_{H^1_0}(-\Delta+q)\}.
    \end{equation*}
We show, for potentials $X_{p,\ka}$ with $p$ satisfying \eqref{eq:p_cond}, the existence of $\qb$ as an element of $\cB_d$ and establish a formula for its Fourier transform that shows, in particular, that $\qb$ is uniquely determined by \eqref{eq:moment_prob_formal} in $\cB_d$ (see also \Cref{main-thm:Fourier Radial}). The Fourier transform of $f\in\cE_\rad'(\R^d)$ is defined as
\begin{equation} \label{eq:Fourier_distribution}
    \F (f)(\xi)  := \br{f,e_{-i\xi}}_{\cE'\times\cC^\infty}, \qquad  e_\zeta(x):=e^{\zeta\cdot x},\qquad  x,\xi \in \R^d,\; \zeta\in\IC^d.
\end{equation}
We denote by $Z_{\ell,d}$, $\ell\in\IN_0$, the \textit{zonal functions}, which are characterized in terms of $\cP_{\ell,d}:L^2(\IS^{d-1})\To\H_\ell^d$, the orthogonal projection onto $\gH^d_\ell$, by 
 \begin{equation*}
    \cP_{\ell,d}u(x)=\int_{\IS^{d-1}}Z_{\ell,d}(x\cdot y)u(y)\, \d{y},\qquad u\in L^2(\IS^{d-1}).  
\end{equation*}
The zonal functions are directly related to the Legendre polynomials; more details are given in \Cref{app:sh}.  
\begin{main-theorem}\label{main-thm:existence_Born}
    Let $\ka\in\R$ and $p$ be admissible. For every $q \in X_{p,\ka}(\IB^d)$
    there exists a unique $\qb \in \cB_d$ such that
    \begin{equation}\label{eq:Born_moments}
        \sigma_\ell[\qb,\kappa] = (\lambda_\ell[q,\kappa]-\lambda_\ell[0,\kappa]) \varphi_\ell(\sqrt{\ka})^2, \qquad \text{for all } \, \ell \in \N_0.
    \end{equation}
    The Fourier transform of this distribution can be obtained from the expression:
    \begin{equation}\label{eq:Born_Fourier_introduction}
         \F (\qb)(\xi)  = 
        (2\pi)^d \sum_{\ell=0}^\infty (\lambda_\ell[q,\kappa] - \lambda_\ell[0,\kappa])\varphi_\ell(\sqrt{\ka})^2  Z_{\ell,d}\left(1-\frac{\abs{\xi}^2}{2\kappa}\right).        
    \end{equation}
    If we further assume  $\kappa \notin  \Spec_{H^1_0}(-\Delta)$ then it also holds that
    \begin{equation}\label{eq:linear_cgos}
        \F (\qb)(\xi) = \hp{\cc{e_{\zeta_1}}}{(\Lambda_{\kappa,q}-\Lambda_{\kappa,0})e_{\zeta_2}}_{L^2(\Sp^{d-1})},        
    \end{equation} 
    with $\zeta_1,\zeta_2\in \C^d$ satisfying 
    $\zeta_1+\zeta_2=-i\xi$ and $\zeta_1\cdot\zeta_1 =\zeta_2\cdot\zeta_2=-\kappa$.
\end{main-theorem}
We will refer to $\qb$ as the \textit{Born approximation} of $q$ at energy $\ka$. Its explicit reconstruction formula \eqref{eq:Born_Fourier_introduction} can be used to numerically approximate $q$; see \Cref{sec:subsec_numerical} below. On the other hand, the identity \eqref{eq:linear_cgos} shows that $\qb$ can also be obtained by complex geometrical optics solutions as was done for $\ka=0$ in \cite{BCMM_2022_Born}.

Our next result shows, among other properties, that $\qb$ always coincides with a locally integrable function outside the origin.   In fact, the Born approximation is as singular as the potential is, except possibly at the origin.

We define the weighted $L^1$-space $L^1_s(\IB^d): = \{ f\in L^1_\loc(\IB^d\setminus\{0\}) : \; \norm{f}_{L^1_s(\IB^d)}<\infty\}$ where $s \in \R$ and 
\begin{equation} \label{eq:norm_L1_wheight}
    \norm{F}_{L^1_s(\IB^d)} := \int_{\IB^d} |F(x)|  |x|^{2(s-\nu_d)} \, \d x.
\end{equation}

\begin{main-theorem}\label{main-thm:basic_properties}
    Let $\ka \in \R $ and $p$ be admissible. For every $q\in X_{p,\ka}(\IB^d)$, there exists an $\ell_q\geq0$ and a real and radial function $\qs \in L^1_{\ell_q}(\IB^d)$ such that $\qb|_{\IB^d\setminus \{0\}} = \qs$. In fact,
    \begin{equation}\label{eq:bdryv}
        \qb - q\in \cC(\ol{\IB^d}\setminus \{0\}),\qquad (\qb - q)|_{\IS^{d-1}} =0,
    \end{equation}
    and $\qb- q$ belongs to $\cC^{m+1}(\ol{\IB^d} \setminus \{0\})$ if $q \in \cC^{m}(\ol{\IB^d}  \setminus \{0\})$  for any $m\in \N_0$.

    If in addition, $p>d/2$, $\ka < (j_{0,1})^2$, and $q-\ka \ge 0$  $a.e.$ in $\IB^d$, then $\qb \in L^1_\rad(\IB^d)$.

\end{main-theorem}
\begin{remark} 
The identity $\qb|_{\IB^d\setminus \{0\}} = \qs$ holds in the sense of distributions and shows that $\qb$ can be identified with a locally integrable function on $\IB^d\setminus\{0\}$. In other words, $\qb$ is a regularization of the singular function $\qs$ in the sense of \cite[Chapter 1]{GelShilVol1}.
Note that \eqref{eq:bdryv} implies that $q|_{\IS^{d-1}}=\qb|_{\IS^{d-1}}$ when $q$ has a well-defined trace on the boundary. See \Cref{prop:q_s_def} for a more precise estimate of the pointwise behavior of $|\qb-q|$ on $\IB^d\setminus\{0\}$.
\end{remark}

We prove two important properties of the map $(\Phi_{\rm B}^\ka)^{-1}$: a strong local form of injectivity and Hölder continuity under uniform $L^p(\IB^d)$ bounds for the potentials.
\begin{main-theorem}\label{main-thm:stability}
     Let $\ka \in \R$, $p$ be admissible and $b\in(0,1)$. Let $U_b:= \{x\in\IB^d: b<|x|< 1\}$.
    \begin{enumerate}[i)]
        \item For every $q_1,q_2\in X_{p,\ka}(\IB^d)$ let $q_{1,\ka}^\mB$ and $q_{2,\ka}^\mB$ the respective Born approximations.  
        Then 
        \begin{equation*}
        q_{1,\ka}^\mB(x) = q_{1,\ka}^\mB(x) \; a.e. \text{ in } \, U_b \; \iff \; q_1(x) = q_2(x)\; a.e. \text{ in } \, U_b.
        \end{equation*}
        \item For every $K>0$ there exist $\varepsilon(b,\ka) >0$ and  $C(b,\ka,K)>0$ such that, for every $q_1,q_2\in  X_{p,\ka}(\IB^d)$ satisfying
        \begin{equation*}
         \max_{j=1,2} \norm{q_j}_{L^p(U_b)}  <  K,\qquad  \left\| q^\mB_{1,\ka}  - q^\mB_{2,\ka}  \right\|_{L^1(U_b)}  <\varepsilon(b,\ka),
        \end{equation*}
        one has
        \begin{equation}\label{est:stability_qB}
         \left\| q_1  - q_2  \right\|_{L^1(U_b)}    < C(b,\ka,K)
         \left(\left\| q^\mB_{1,\ka}  - q^\mB_{2,\ka}  \right\|_{L^1(U_b)}\right)^{1/(p'+1)},   
        \end{equation}
     where $p'$ is the Hölder conjugate exponent of $p$.
    \end{enumerate}
\end{main-theorem}

We finally show that in the high-energy limit $\ka\to-\infty$, the Born approximation at energy $\ka$ recovers the potential.
\begin{main-theorem}\label{main-thm:limit}
    Let $q\in L^\infty(\IB^d,\IR)$, $\kappa\in\IR$, $\xi\in\IR^d$, and $\zeta_1,\zeta_2\in \C^d$ be such that $\zeta_1+\zeta_2=-i\xi$, and $ \zeta_1\cdot\zeta_1 =\zeta_2\cdot\zeta_2=-\kappa$.
    Then, locally uniformly in $\xi$,
    we have
    \begin{equation} \label{eq:no_radial_lim}
      \lim_{\kappa\to-\infty}\hp{\cc{e_{\zeta_1}}}{(\Lambda_{q,\kappa}-\Lambda_{0,\kappa})e_{\zeta_2}}_{L^2(\Sp^{d-1})}=\F q(\xi).  
    \end{equation}
    In particular, whenever $q\in L^\infty_\rad(\IB^d,\IR)$, in which case $\qb$ is always defined when $\ka\leq -\norm{q}_{L^\infty(\IB^d)}$, we have
    \begin{equation}\label{eq:semiclass}
        \lim_{\kappa\to-\infty}\F \qb(\xi)=\F q(\xi),\qquad \text{for all }\xi\in\R^d.     
    \end{equation}
\end{main-theorem}
\begin{remark} \
    \begin{enumerate}[i)]
        \item Notice that the statement \eqref{eq:no_radial_lim} is valid even if $q$ is not radial. Moreover, it will be clear from the proof of \eqref{eq:no_radial_lim} that the analogous result holds for any bounded domain $\Omega\subseteq\R^d$ with $\cC^1$ boundary.
        \item The convergence in \eqref{eq:semiclass}, which involves entire functions, takes place in Fourier space. At this moment, it is not clear that this convergence can be expressed in a natural way in physical space; however, this should be possible at least for some classes of potentials. 
        \item The following is proved in \Cref{prop:spectrum}:
        \begin{equation*}
            \lim_{\ka \to -\infty}\Tr |\Phi^\ka(q)- \d \Phi^\ka_0(q)|=\lim_{\ka \to -\infty}\Tr |\Lambda_{q,\ka}-\Lambda_{0,\ka}-\d \Phi^\ka_0(q)|=0.
        \end{equation*}
        This relation complements \eqref{eq:semiclass}, to actually show that the non-linear map $\Phi^\ka$ is asymptotically linear in the high-energy limit and that $\Phi_{\rm B}^\ka$ converges to the identity as $\ka\to-\infty$.
    \end{enumerate}
\end{remark}
The local uniqueness phenomenon proved in the first part of \Cref{main-thm:stability} is illustrated in \cref{Fig:4}, whereas the phenomenon of recovery of singularities in \Cref{main-thm:basic_properties} is illustrated in \cref{Fig:1}, \cref{Fig:3} and \cref{Fig:4}. Evidence for the convergence of the Born approximations to the potential in physical space is presented in \cref{Fig:3} below.

Theorems \ref{main-thm:existence_Born} to \ref{main-thm:limit} are proved in \Cref{sec:proofs}. The analogues of Theorems \ref{main-thm:existence_Born}, \ref{main-thm:basic_properties} and \ref{main-thm:stability} in the case $\ka=0$ have been proved in \cite{Radial_Born}. These results can be recovered from the ones in this work in the limit $\ka\to 0$. For instance, the classical \textit{Hausforff moments}, which played the same role as the $\ka$-moments here, can be obtained from $\sigma_\ell[q,\ka]$ by
\begin{equation} \label{eq:limit_res}
    4^{\ell+\nu_d} \Gamma(\ell+d/2)^2\lim_{\ka\to 0}\frac{\sigma[q,\ka]}{\ka^\ell }=\frac{1}{|\IS^{d-1}|} \int_{\IB^d} q(x)  |x|^{2\ell}  \, \d x,  \qquad \forall \ell \in \IN_0.
\end{equation}
The proofs of the main results in \cite{Radial_Born} rely on tools from the inverse spectral theory of Schrödinger operators on the half-line, mainly results involving the notion of $A$-amplitude of a Schrödinger operator introduced by Simon in \cite{Simon_spectral_I}, and further developed in \cite{Simon_spectral_II, Simon_spectral_III, Remling03, Avdonin_Mikhay_Rybkin_07}, among many other works. In our case, these tools are not directly applicable, which leads us to introduce and prove the existence of an $A_\ka$-amplitude that encodes spectral information in a similar way as Simon's $A$-amplitude does, but is better adapted to our setting. The main results on the $A_\ka$-amplitude are stated in \Cref{sec:subsec_1dstate} and proved in \Cref{sec:subsec_1d}.

We note that the approach based on one-dimensional inverse spectral theory has been applied in the context of the Steklov problem for warped
product manifolds in \cite{DKN21_stability_steklov, DHN21, DKN23}.
In particular, the results in \cite{DKN21_stability_steklov}   imply stability and uniqueness results for the radial Calderón problem, both for the conductivity and Schrödinger cases.
We also mention that spectral theory methods had already been used in the context of the radial Calderón problem in \cite{KoVo85,Syl_1992}.

\subsection{Numerical reconstruction} \label{sec:subsec_numerical}

In this subsection, we show the capabilities of the Born approximation as an effective tool to approximate the potential; numerical methods based on this strategy are described in \cite{BCMM_2024_numerical} in the case $\ka=0$. The key remark is that the Fréchet differential $\d \Phi^\ka_0$ coincides with the differential of $\Phi^0$ around the constant function $-\ka\in\IR\setminus\Spec_{H^1_0}{\Delta}$, and that $\Lambda_{q,0}=\Lambda_{q+\kappa,\kappa}$. Therefore, for such $\ka$, given $q\in L^\infty_\rad(\IB^d,\IR)$ such that $0$ is not a Dirichlet eigenvalue of $-\Delta+q$, one should have:
\begin{equation*}
    q\approx -\ka + \d \Phi^0_{-\ka}(\Lambda_{q,0} - \Lambda_{0,\ka}) = -\ka + (q+\kappa)^{\mathrm{B}}_\kappa.
\end{equation*}
When $\ka$ is chosen appropriately, so that $q +\ka$ is small in some norm, one can expect this function to be well-approximated by $(q+\kappa)^{\mathrm{B}}_\kappa$. 
Recall that, by \Cref{main-thm:existence_Born}, the Born approximation $(q+\kappa)^{\mathrm{B}}_\kappa$ is constructed solely in terms of $\Spec(\Lambda_{q,0})$ (the dependence on $\ka$ being explicit). 
A possible choice of $\ka$ that can be implemented numerically consists in ensuring that $\int_{\IB^d}(q+\kappa)^{\mathrm{B}}_{\kappa} \d x=\F(q+\kappa)^{\mathrm{B}}_{\kappa}(0)$, which is a quantity that depends only on $\Spec(\Lambda_{q,0})$ and $\kappa$, is the smallest possible. 
It is therefore natural to choose $\ka=\kappa_*$ defined by the equation 
$$\F(q+\kappa_*)^{\mathrm{B}}_{\kappa_*}(0)=0.$$
Numerical experiments suggest that if $\kappa_*$ exists, it can be obtained by the iterative fixed point scheme
$$\kappa_{n+1}=\kappa_n-\frac{\F(q+\kappa_{n})^{\mathrm{B}}_{\kappa_n}(0)}{|\IB^d|}.$$

We illustrate these ideas in \cref{Fig:1,Fig:2} where we plot the radial profiles of $q:\IB^3\to\R$ and the corresponding $(q+\kappa)^{\mathrm{B}}_\kappa-\kappa$ for
$$\kappa=0,\qquad\kappa=-q|_{\IS^{d-1}},\qquad\kappa=-\frac{1}{2}\left(\sup_{\IB^3} q+ \inf_{\IB^3} q\right),\qquad\kappa=\kappa_*.$$
In the caption of each plot, we include the $L^1(\IB^3)$ norm of the difference of the functions plotted. Notice that in all eight plots, except for \cref{Fig:2} (A) where the Gibbs phenomenon is too strong, $(q+\kappa)-(q+\kappa)^{\mathrm{B}}_\kappa$ is a continuous function outside $x=0$ that vanishes at $\IS^{2}$, as stated in \Cref{main-thm:basic_properties}.

The Julia code used to produce these figures can be found in \cite{JuliaCode}.

\begin{figure}[h]
        \centering
        \begin{subfigure}[b]{0.475\textwidth}
            \centering
            \includegraphics[width=0.9\textwidth]{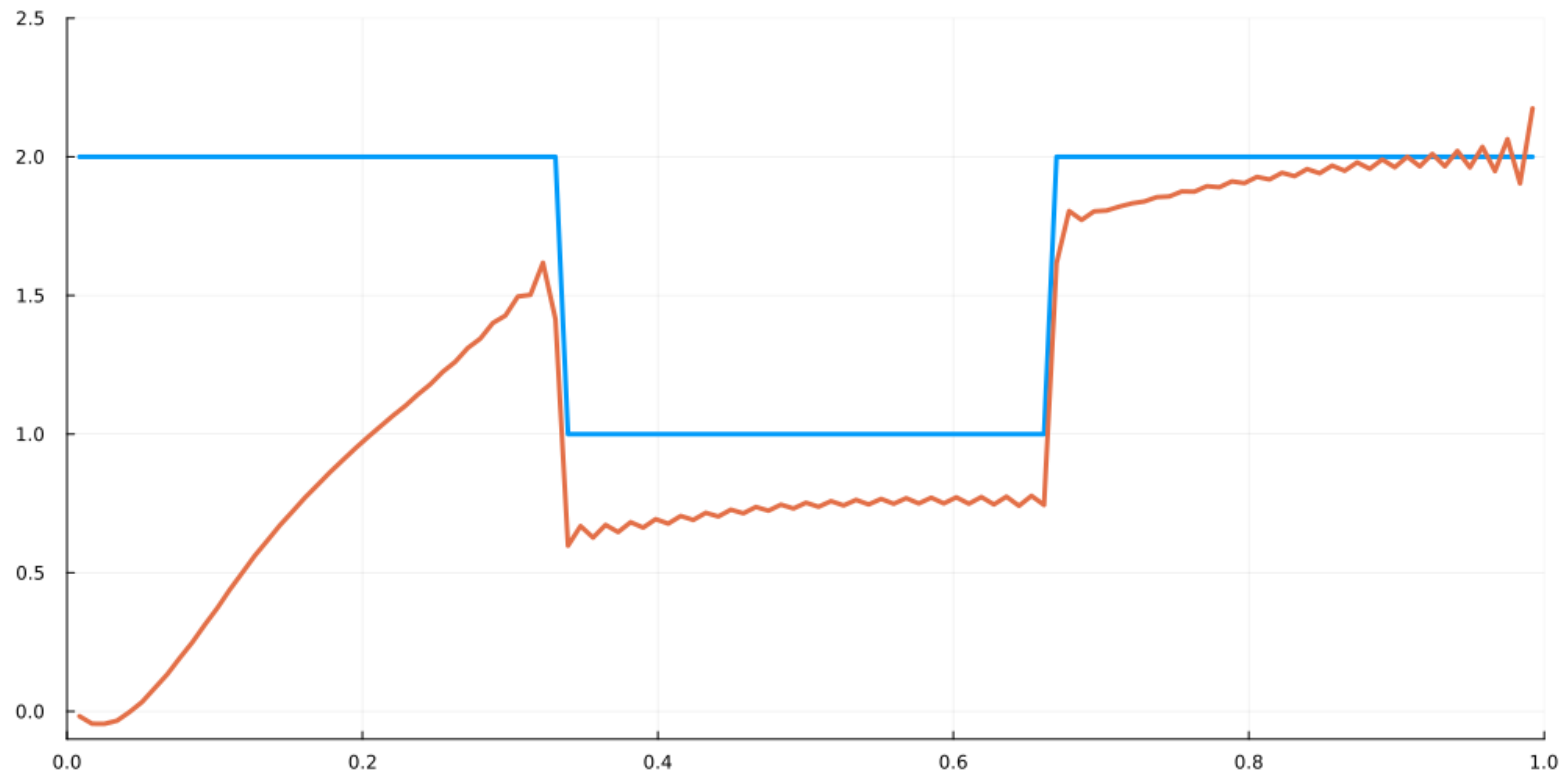}
            \caption{$\ka=0$, $\norm{\,\cdot\,}_{L^1(\IB^3)}=0.66066$.}  
        \end{subfigure}
        \hfill
        \begin{subfigure}[b]{0.475\textwidth}  
            \centering 
            \includegraphics[width=0.9\textwidth]{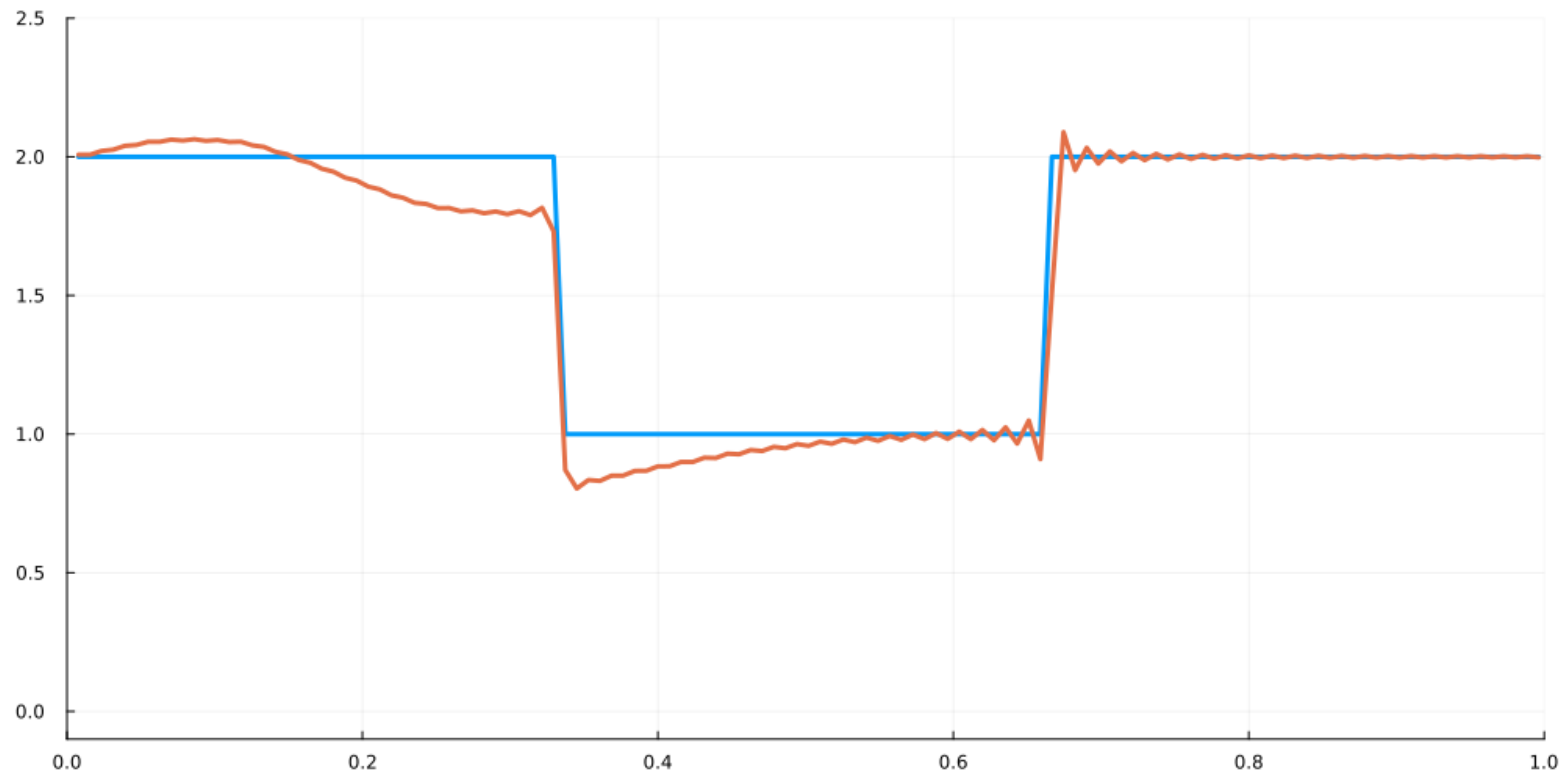}
            \caption{$\ka=-2$, $\norm{\,\cdot\,}_{L^1(\IB^3)}=0.12155$.}
        \end{subfigure}
        \vskip\baselineskip
        \begin{subfigure}[b]{0.475\textwidth}   
            \centering 
            \includegraphics[width=0.9\textwidth]{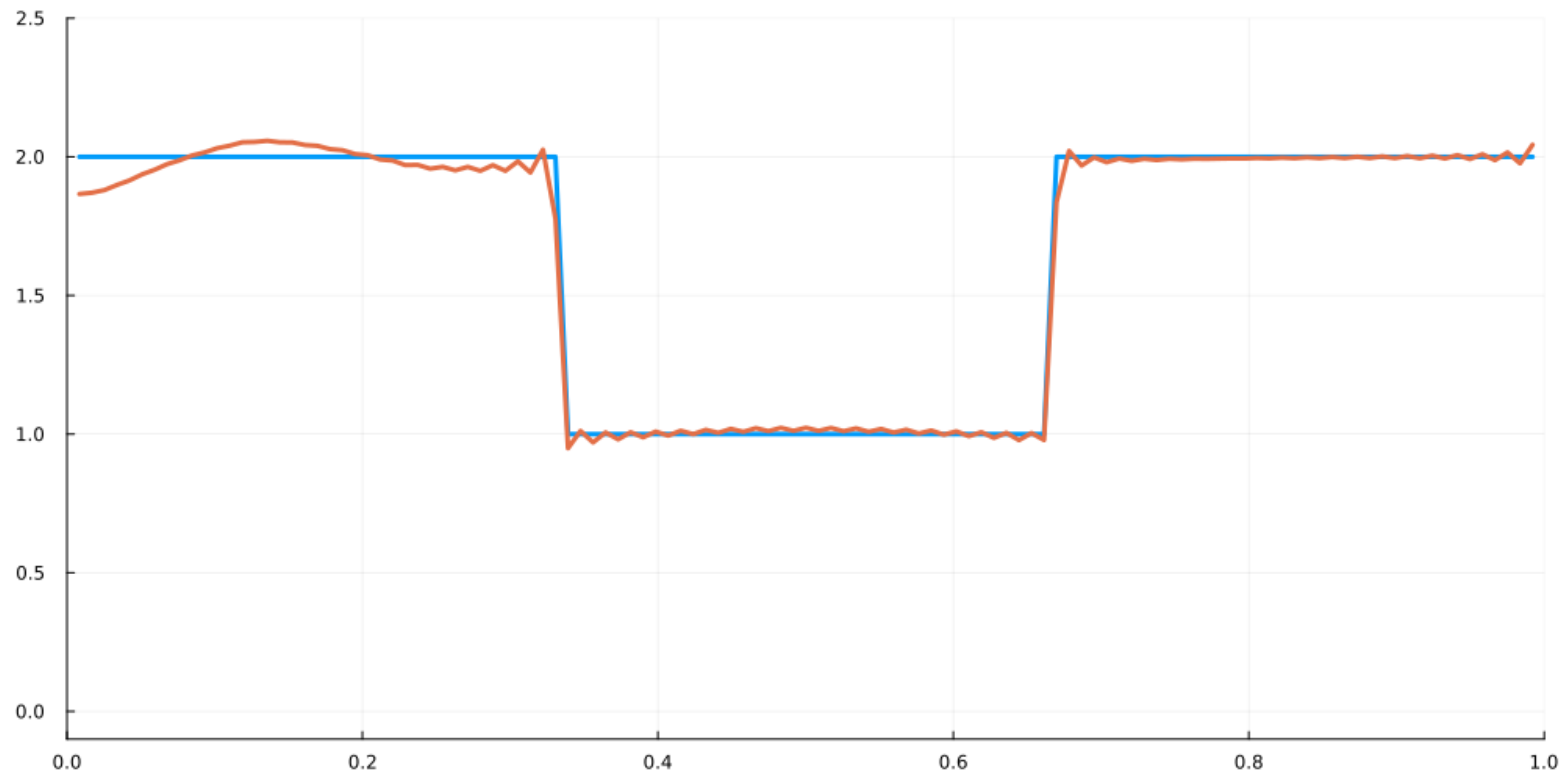}
            \caption{$\ka=-1.5$, $\norm{\,\cdot\,}_{L^1(\IB^3)}=0.05551$.}
        \end{subfigure}
        \hfill
        \begin{subfigure}[b]{0.475\textwidth}   
            \centering 
            \includegraphics[width=0.9\textwidth]{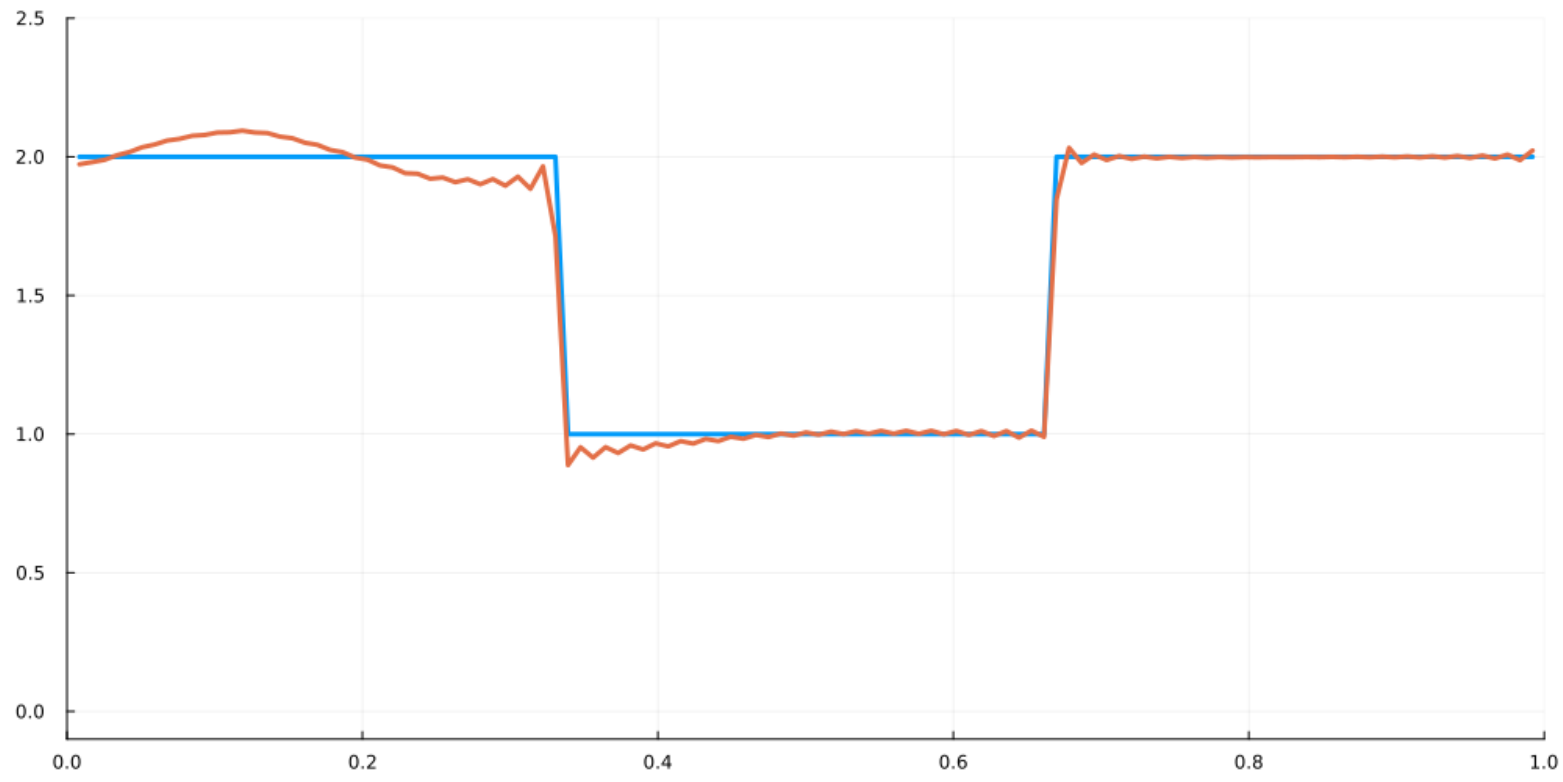}
            \caption{$\ka=\ka_*=-1.73614$, $\norm{\,\cdot\,}_{L^1(\IB^3)}=0.05084$.}
        \end{subfigure}
        \caption{Plots of $q(x)=2-\chi_{(\frac{1}{3},\frac{2}{3})}(\abs{x})$ (blue) and $(q+\kappa)^{\mathrm{B}}_\kappa-\kappa$ (orange).}
        \label{Fig:1}
\end{figure}

\clearpage

\begin{figure}[t]
        \centering
        \begin{subfigure}[b]{0.475\textwidth}
            \centering
            \includegraphics[width=0.9\textwidth]{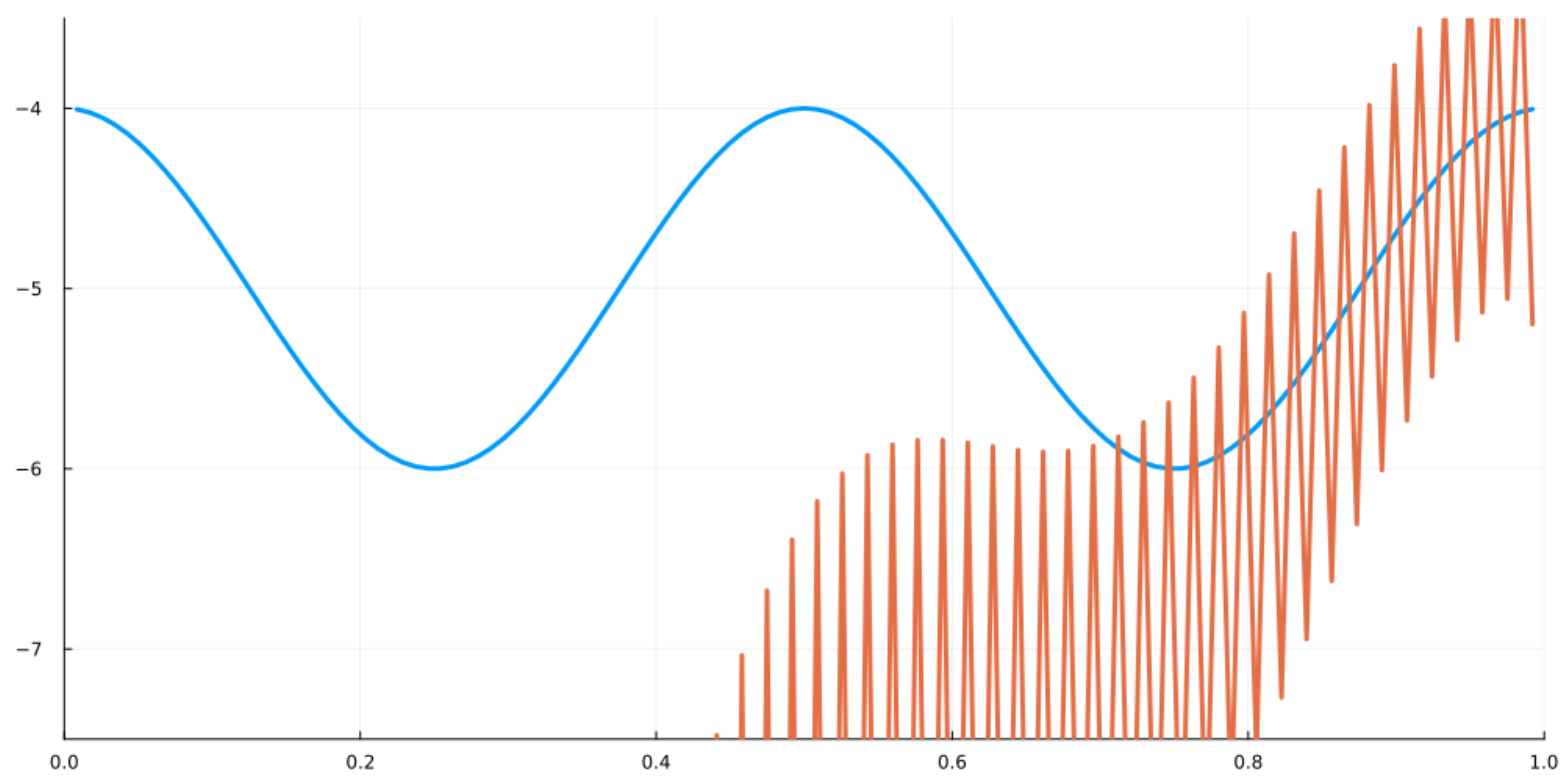}
            \caption{$\ka=0$, $\norm{\,\cdot\,}_{L^1(\IB^3)}=16.4772$.} 
        \end{subfigure}
        \hfill
        \begin{subfigure}[b]{0.475\textwidth}  
            \centering 
            \includegraphics[width=0.9\textwidth]{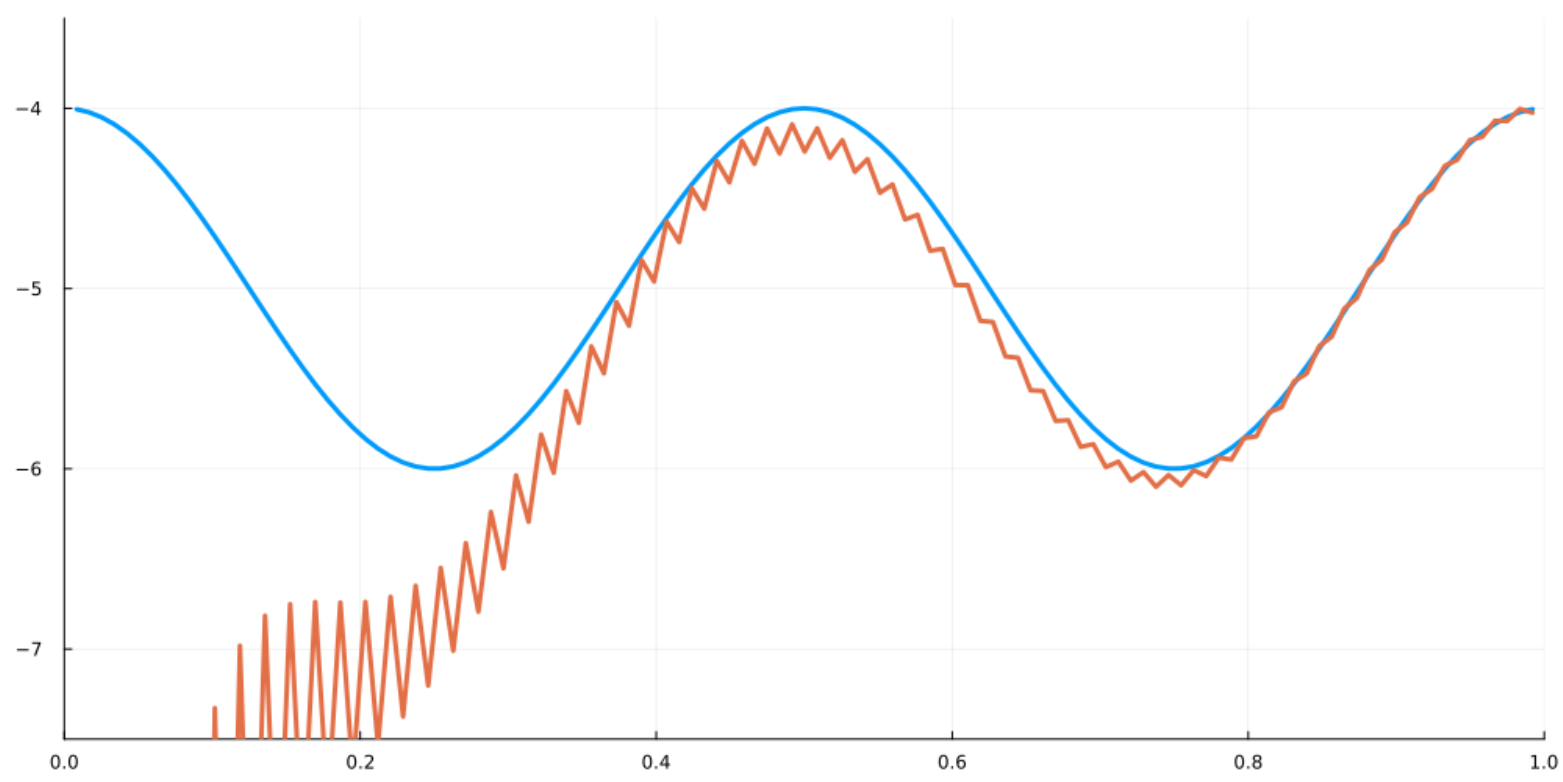}
            \caption{$\ka=4$, $\norm{\,\cdot\,}_{L^1(\IB^3)}=0.51683$.}
        \end{subfigure}
        \vskip\baselineskip
        \begin{subfigure}[b]{0.475\textwidth}   
            \centering 
            \includegraphics[width=0.9\textwidth]{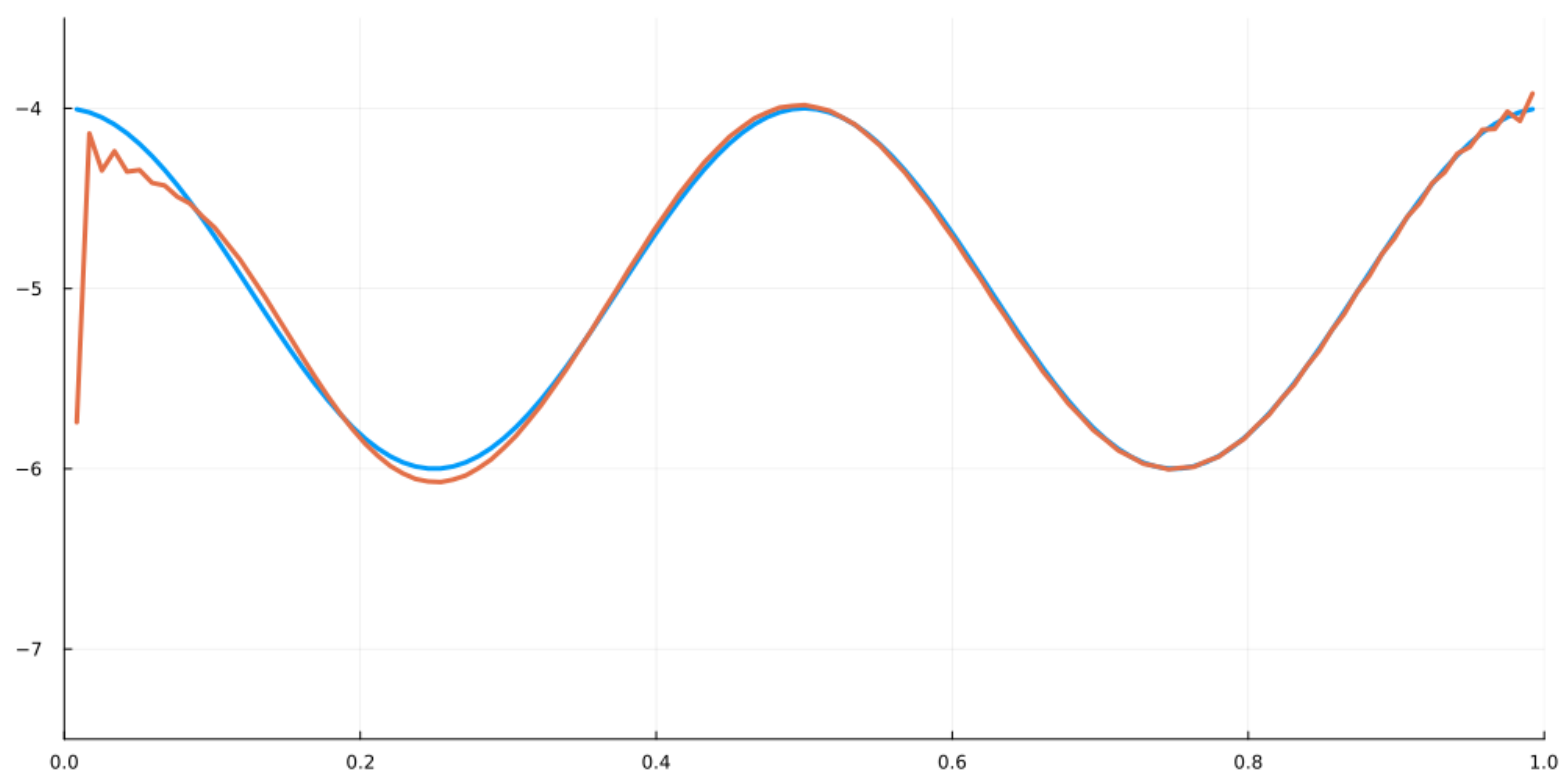}
            \caption{$\ka=5$, $\norm{\,\cdot\,}_{L^1(\IB^3)}=0.07038$.}
        \end{subfigure}
        \hfill
        \begin{subfigure}[b]{0.475\textwidth}   
            \centering 
            \includegraphics[width=0.9\textwidth]{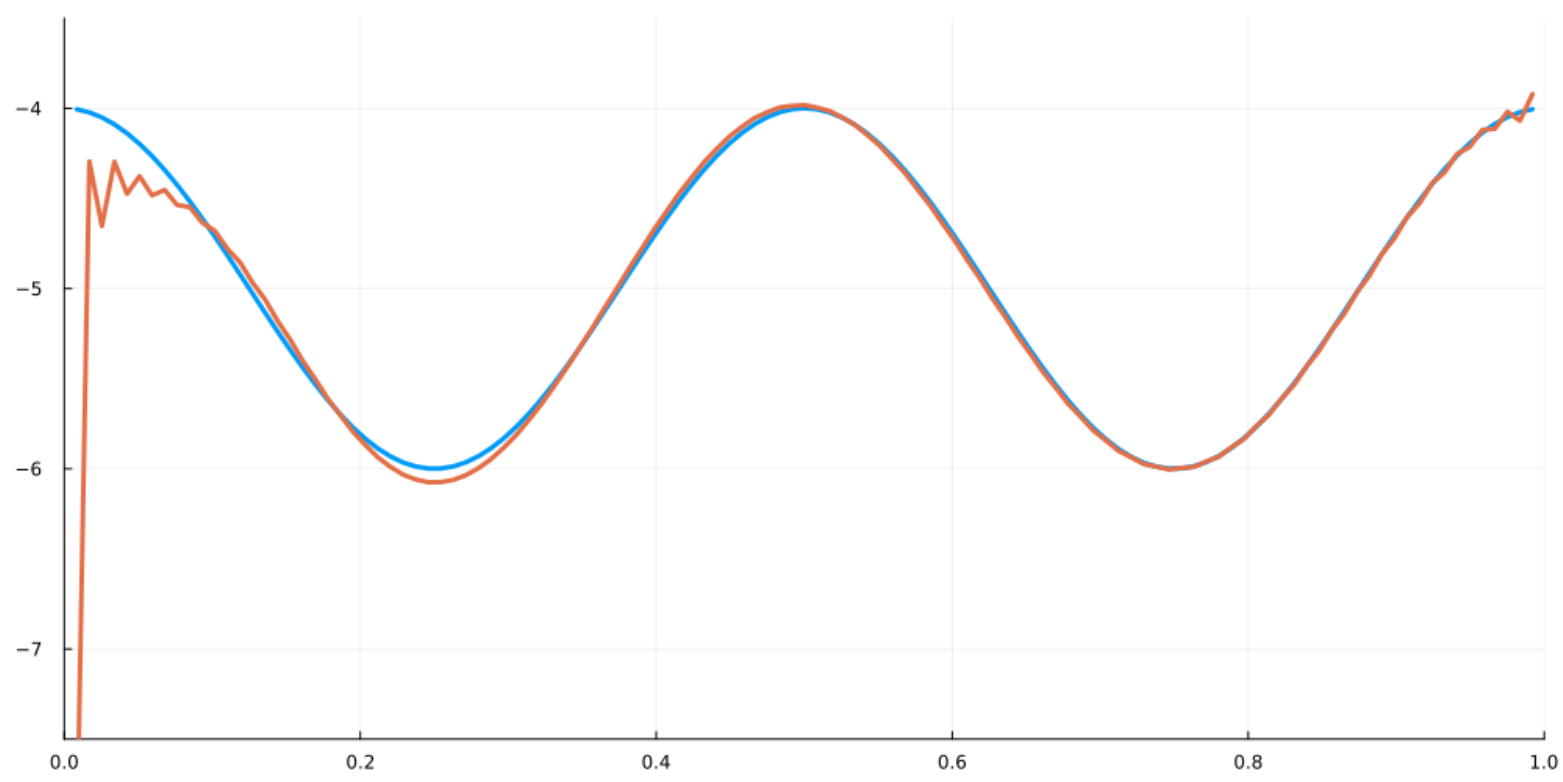}
            \caption{$\ka=\ka_*=4.96727$, $\norm{\,\cdot\,}_{L^1(\IB^3)}=0.07200$.}
        \end{subfigure}
        \caption{Plots of $q(x)=\cos(4\pi\abs{x})-5$ (blue) and $(q+\kappa)^{\mathrm{B}}_\kappa-\kappa$ (orange).}
        \label{Fig:2}
\end{figure}

\begin{figure}[b]
        \centering 
        \begin{subfigure}[b]{0.475\textwidth}
            \centering
            \includegraphics[width=0.9\textwidth]{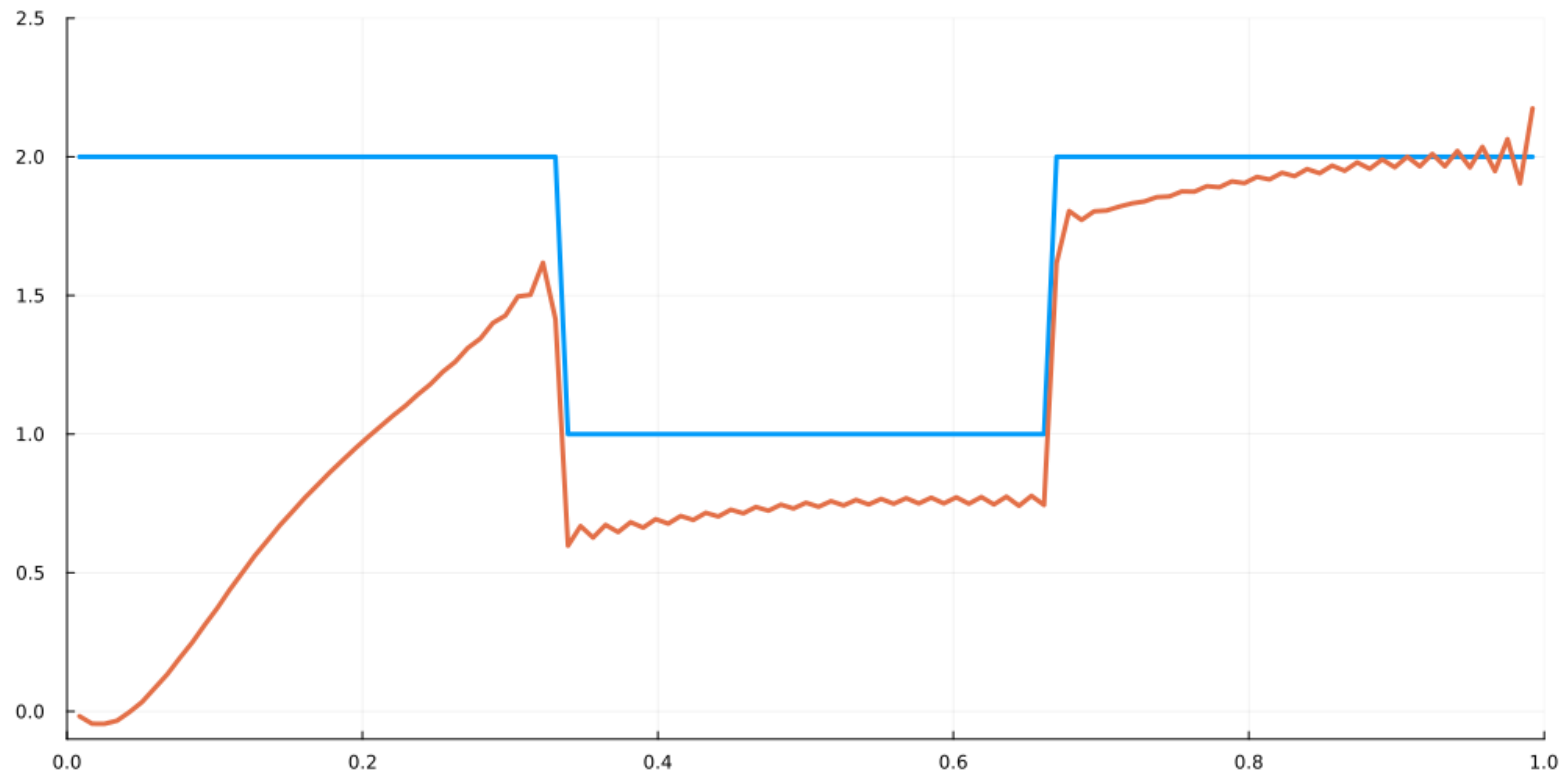}
            \caption{$\ka=0$, $\norm{\,\cdot\,}_{L^1(\IB^3)}=0.66066$.} 
        \end{subfigure}
        \hfill
        \begin{subfigure}[b]{0.475\textwidth}  
            \centering 
            \includegraphics[width=0.9\textwidth]{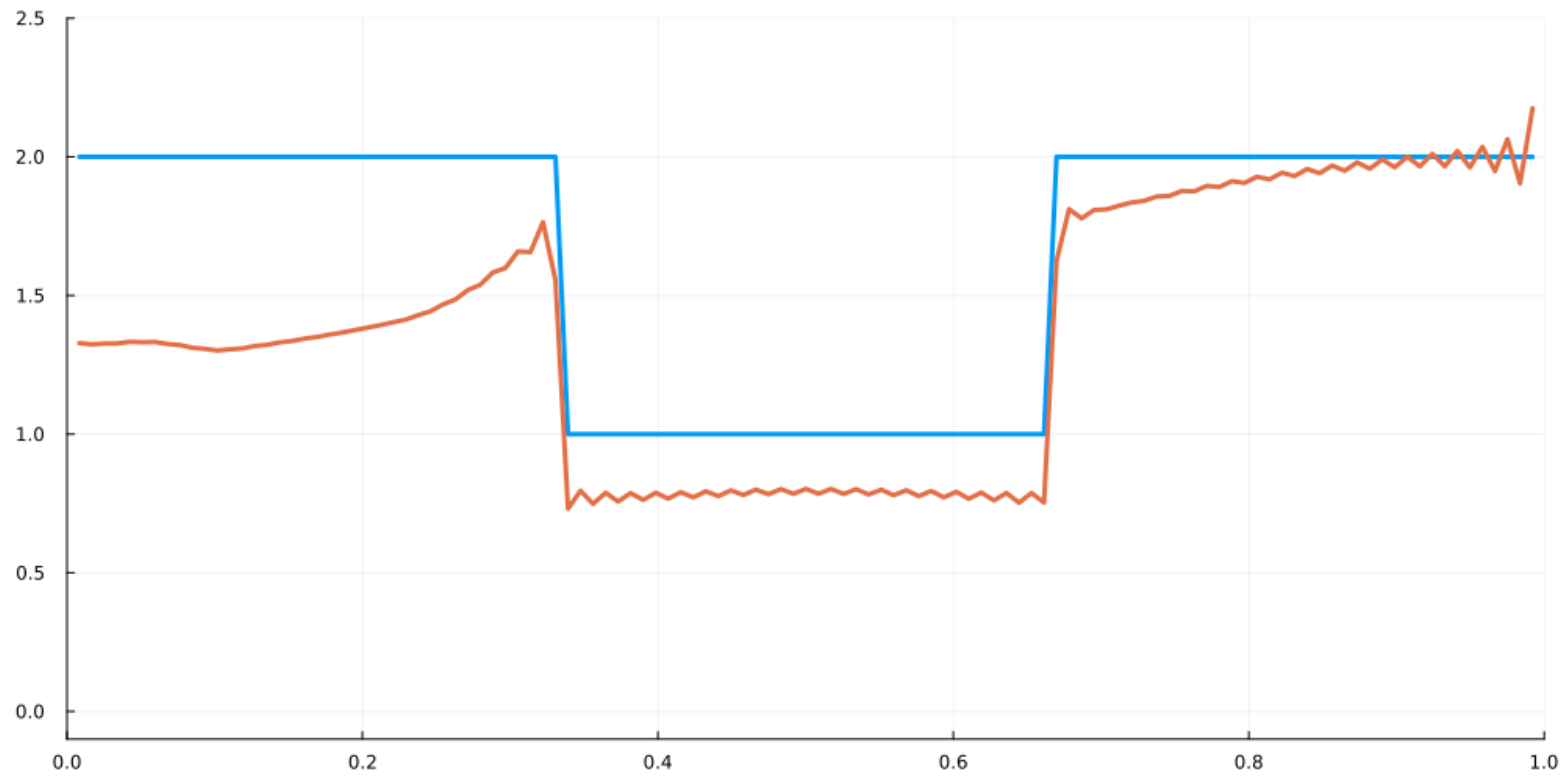}
            \caption{$\ka=-10$, $\norm{\,\cdot\,}_{L^1(\IB^3)}=0.56134$.}
        \end{subfigure}
        \vskip\baselineskip
        \begin{subfigure}[b]{0.475\textwidth}   
            \centering 
            \includegraphics[width=0.9\textwidth]{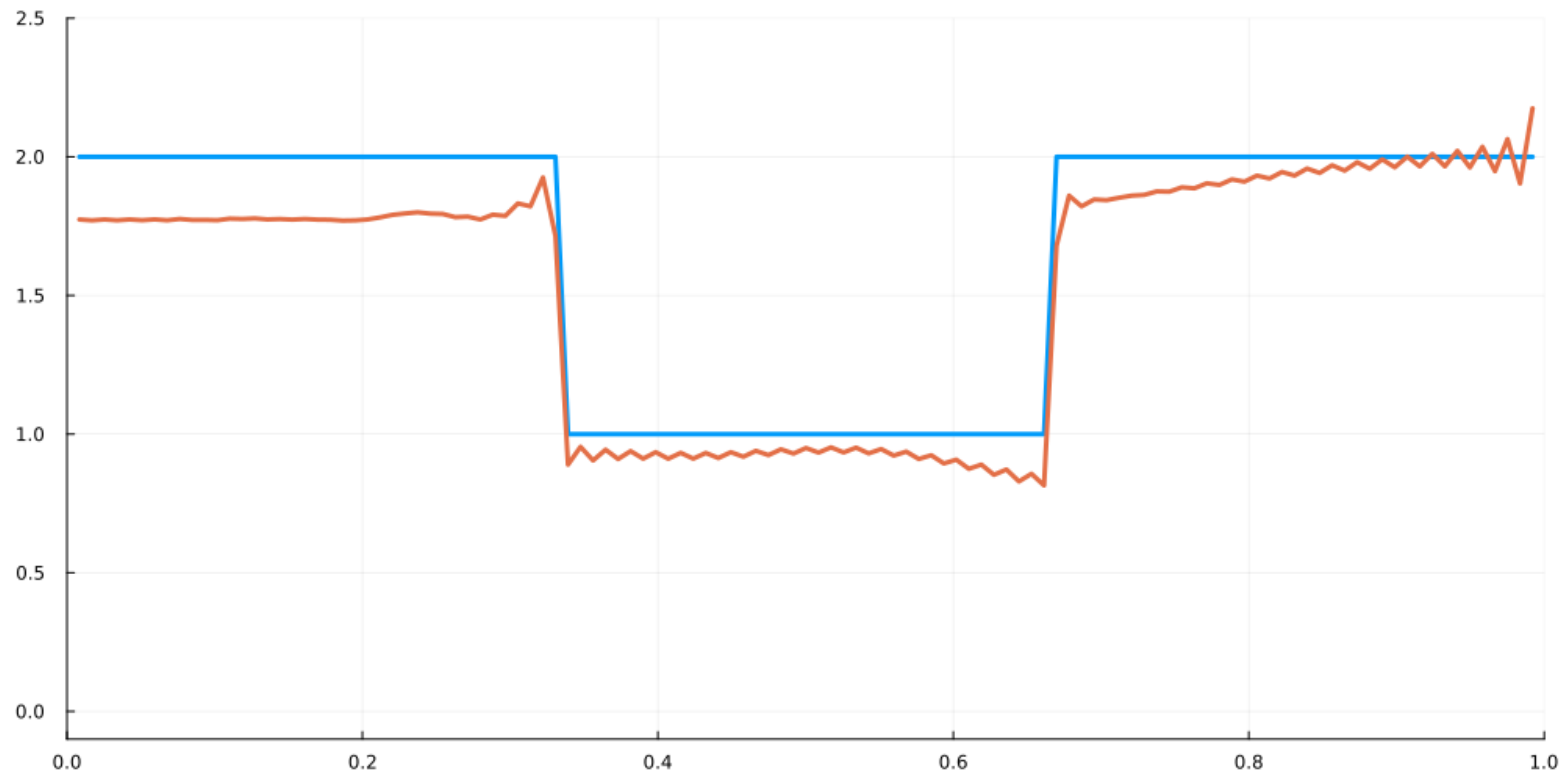}
            \caption{$\ka=-100$, $\norm{\,\cdot\,}_{L^1(\IB^3)}=0.35731$.}
        \end{subfigure}
        \hfill
        \begin{subfigure}[b]{0.475\textwidth}   
            \centering 
            \includegraphics[width=0.9\textwidth]{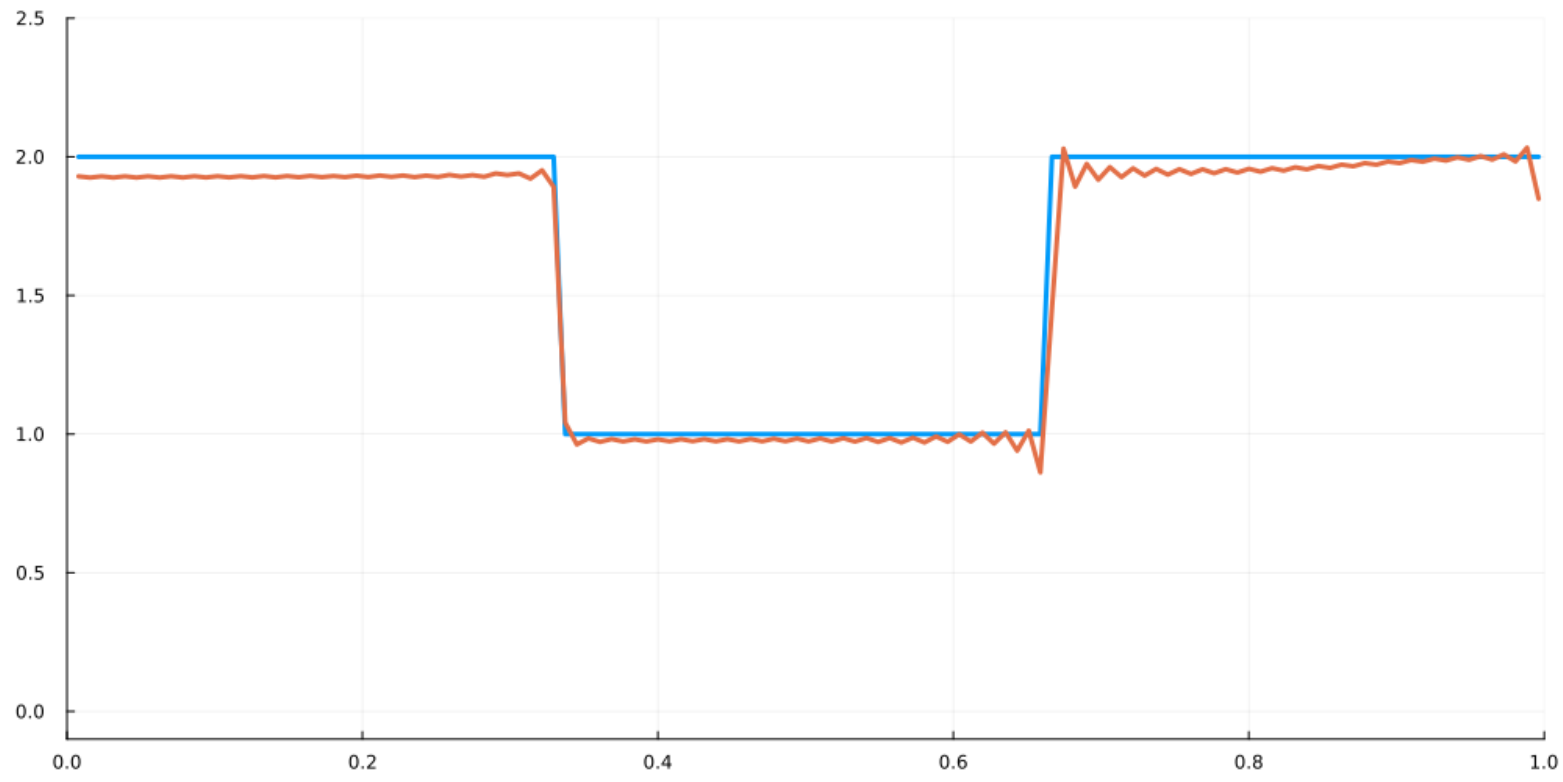}
            \caption{$\ka=-1000$, $\norm{\,\cdot\,}_{L^1(\IB^3)}=0.17218$.}
        \end{subfigure}
        \caption{Plots of $q(x)=2-\chi_{(\frac{1}{3},\frac{2}{3})}(\abs{x})$ (blue) and $\qb$ (orange).}
                        \vspace{17mm}
        \label{Fig:3}
\end{figure}

\clearpage

\begin{figure}[t]
    \centering
    \includegraphics[width=0.9\textwidth]{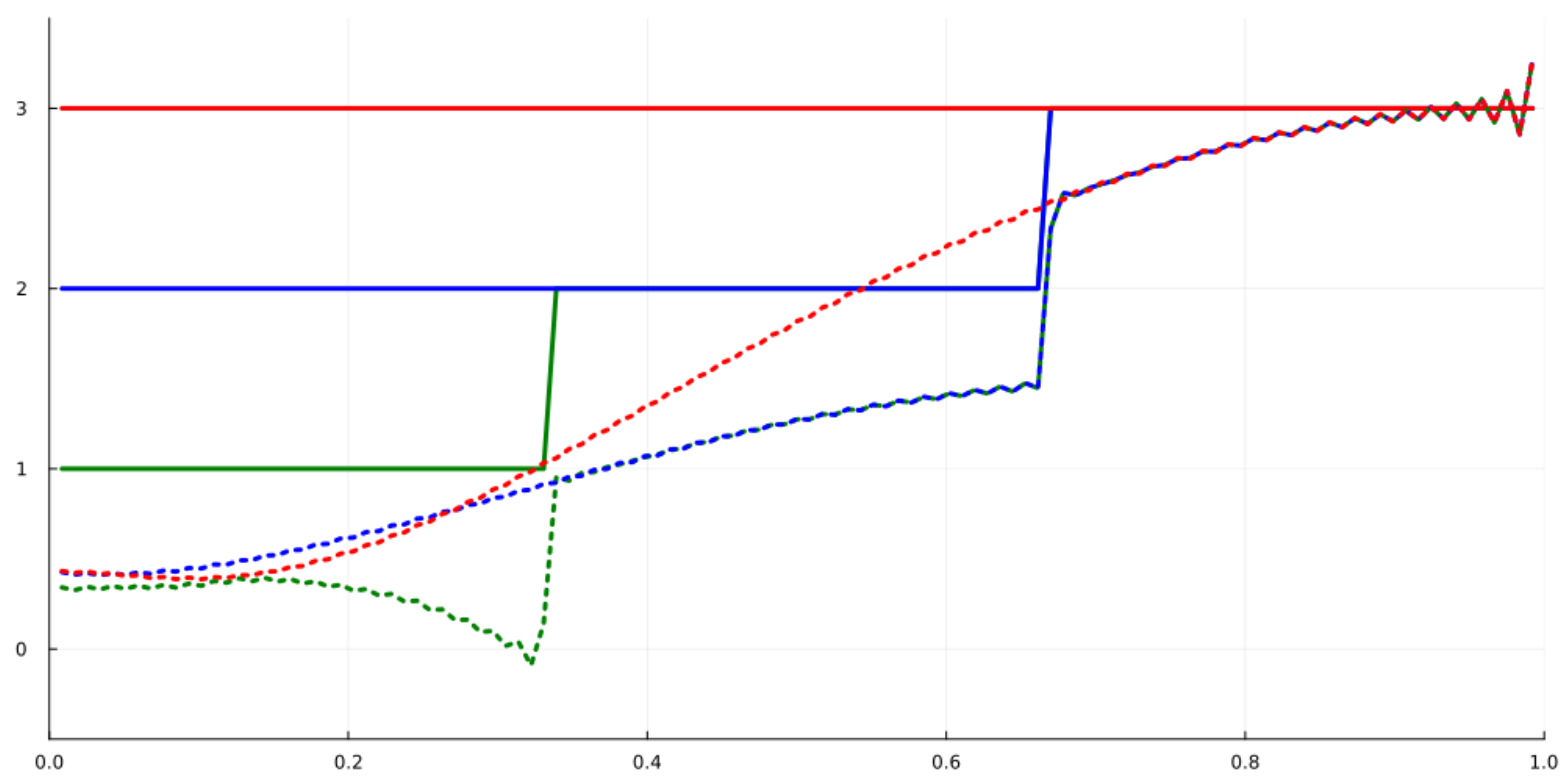}
    \caption{Plots of $q_i$ (solid) of $q_3(x)\coloneqq 3$, $q_2(x)\coloneqq q_3(x)-\chi_{(0,\frac{2}{3})}(\abs{x})$, $q_1(x)\coloneqq q_2(x)-\chi_{(0,\frac{1}{3})}(\abs{x})$ and their respective Born approximations $q_{i,\ka}^\mathrm{B}$ (dashed) at $\kappa=-1$.}
    \label{Fig:4}
\end{figure}

\subsection*{Acknowledgments}
This research has been supported by the grant PID2021-124195NB-C31 from the Agencia Estatal de Investigación (Spain). The authors thank Thierry Daudé and François Nicoleau for many insightful conversations on the inverse spectral theory of Schrödinger operators on the half-line.

\section{Radial Dirichlet-to-Neumann maps} \label{sec:dtn}

We start by deriving some simple properties of the DtN map associated to a radial potential $q$. Let $\kappa\in \IC$ and consider the Dirichlet problem:
\begin{equation}\label{e:kbvp}
    \left\{
\begin{array}{rll}
-\Delta v + qv -\kappa v &=0    \quad\text{ in }   \IB^d,\\
v|_{\IS^{d-1}}  &=f\in\cC^\infty(\IS^{d-1}).
\end{array}\right.
\end{equation}
This problem has a unique solution provided that $q\in L^p_\rad(\IB^d,\IR)$, with $p$ admissible, and $\kappa$ is not a Dirichlet eigenvalue of $-\Delta+q$. In addition, the radiality of $q$ implies that
\begin{equation*}
        \cE_{\ell}:=L^2_\rad(\IB^d)\otimes \H_\ell^d\subset L^2(\IB^d),
    \end{equation*}
is an invariant subspace of $-\Delta+q$. \footnote{One uses the standard identification $
(v\otimes f)(x)=v_\rad(|x|)f(\widehat{x})$, where $\widehat{x}:=x/|x|$ for every $x\in \R^d\setminus\{0\}$.
}
More precisely, the following result holds.
\begin{lemma}\label{l:sepvar}
    Let $p$ be admissible and $q\in L^p_\rad(\IB^d)$. Let $\ell\in\IN_0$ 
    \[
    \ka \in \IC \setminus \Spec_{H^1_0}((-\Delta+q)|_{\cE_{\ell}}),
    \]
    and $f\in\gH^d_\ell$. Then \eqref{e:kbvp} has a unique solution $v\in H^1(\IB^d)$, that is of the form
    \begin{equation}\label{e:sepvar}
    v(x)=b_\ell(|x|,\kappa)f(\widehat{x}), \qquad \forall x\in\IB^d.
    \end{equation} 
    In particular, when $q=0$, $\Spec_{H^1_0}(-\Delta|_{\cE_{\ell}})=\{j_{\ell+\nu_d,k}^2\,:\,k\in\IN\}$ where $(j_{\ell+\nu_d,k})_{k\in\IN}$ is the set of zeros of $J_{\ell+\nu_d}$ and
    \begin{equation}\label{e:defb}
    b_\ell(r,\ka)=\frac{\varphi_{\ell}(\sqrt{\kappa} r)}{\varphi_{\ell}(\sqrt{\kappa})},\qquad r\in\IR_+.
    \end{equation} 
    Moreover, when $\ka \leq 0 $ one has
    \begin{equation} \label{eq:b_ell_asymp}
         \int_{\IB^d}|\varphi_{\ell}(\sqrt{\kappa} r)|^2\,\d x \leq \frac{|\IS^{d-1}||\varphi_{\ell}(\sqrt{\kappa})|^2}{2(\ell+\nu_d+1)}.
    \end{equation}
\end{lemma}
\begin{proof}
    After writing $q(x)=q_0(|x|)$ and changing \eqref{e:kbvp} to polar coordinates, one finds that, if a function in $H^1(\IB^d)$ of the form \eqref{e:sepvar} is a solution of \eqref{e:kbvp}, then $b_\ell$ must solve
    \begin{equation} \label{eq:radial_sol_b_libre}
    -\frac{1}{r^{d-1}}\frac{\d}{\d r}\left(r^{d-1} \frac{\d}{\d r} b_\ell(r,\kappa)\right)+ \left(\frac{\ell(\ell+d-2)}{r^2}+q_0(r)-\kappa\right)b_\ell(r,\kappa)=0,\quad r\in (0,1),
    \end{equation}
    with $b_\ell(1,\ka)=1$. The Liouville change of variables $w_\ell(t):=b_\ell(e^{-t},\ka)e^{-\nu_d t}$ (see, for instance, \cite[Chapter 10.9]{MR972977}) transforms \eqref{eq:radial_sol_b_libre} into a one-dimensional Schrödinger equation on the half-line:
    \[
    (-\partial_t^2+Q(t))w_\ell(t)=-(\ell+\nu_d)^2w_\ell(t),\qquad Q(t)=e^{-2 t}(q_0(e^{-t})-\ka),\quad t\in\IR_+,
    \]
    with $w_\ell(0)=1$, which is uniquely solvable in $L^2(\IR_+)$. See \Cref{lemma:DtN_mQ} for a proof of this and \Cref{lemma:m_function_ka} for a proof of \eqref{e:defb}. 
    We now prove \eqref{eq:b_ell_asymp}. First, note that this is an identity when $\ka=0$. For $\kappa <0$, \eqref{eq:defvphi} implies
    \begin{equation*}
        \int_{\IB^d}\left|\frac{\varphi_{\ell}(\sqrt{\kappa} |x|)}{\varphi_{\ell}(\sqrt{\kappa})}\right|^2\,\d x =|\IS^{d-1}|\int_0^1\left|\frac{I_{\ell+\nu_d}(\sqrt{|\kappa|}r)}{I_{\ell+\nu_d}(\sqrt{|\kappa|})}\right|^2 r \,\d r.
    \end{equation*}
    The conclusion follows from the identity
    \begin{equation*}
        \frac{I_\nu(r)}{I_\nu(s)}<\left(\frac{r}{s}\right)^\nu,
    \end{equation*}
    valid as soon as $\nu>-1/2$ and $0<r<s$ (see \cite{Paris84}). 
\end{proof}

\begin{remark}\label{rem:symb} \
    \begin{enumerate}[i)]
        \item By uniqueness of solutions of \eqref{e:kbvp} with Dirichlet condition $ b_\ell(1,\kappa)=1$
        one necessarily must have:
        \begin{equation}\label{eq:conjjj}
        b_\ell(\cdot,\cc\kappa)=\cc{b_\ell(\cdot,\kappa)}.
        \end{equation}
        \item  Identity \eqref{e:sepvar} implies 
        \begin{equation}\label{e:eigenvrough}
        \Lambda_{q,\kappa} f = \partial_r b_\ell(1,\kappa) f,\qquad \forall f\in\gH^d_\ell.
        \end{equation}
        In other words, the spectrum of $\Lambda_{q,\kappa}$ consists of the eigenvalues:
        \begin{equation*}
            \lambda_\ell[q,\ka]= \partial_r b_\ell(1,\kappa),\qquad \forall\ell\in\IN_0.   
        \end{equation*}
        \item  When $\kappa\in\IR$, it is possible to give a meaningful definition of the eigenvalues $\lambda_\ell[q,\kappa]$ for every $\ell$ large enough, even when $\kappa$ is a Dirichlet eigenvalue of $-\Delta+q$, see \cite[Definition~2.4]{Radial_Born}. In \cite[Lemma~2.1]{Radial_Born} it is shown that \eqref{e:kbvp} has a unique solution of the form \eqref{e:sepvar} that belongs to $H^1_{\rm loc}(\IB^d\setminus\{0\})\cap L^2_{\nu_d-1/2}(\IB^d)$.
    \end{enumerate}
\end{remark}

Integration by parts shows the following important result, known as Alessandrini's identity.
\begin{lemma}\label{lemma:alessss}
    Let $p$ be admissible and $\ka\in\IC\setminus \Spec_{H^1_0}(-\Delta)$. If $q\in L^p(\IB^d,\IR)$ is such that $\ka$ is not a Dirichlet eigenvalue of $-\Delta+q$ then, for every $f,g\in L^2(\IS^{d-1})$,
    \begin{equation}\label{e:alessandrini1}
    \hp{g}{(\Lambda_{q,\kappa} - \Lambda_{0,\kappa}) f}_{L^2(\IS^{d-1})} = \int_{\IB^d} q(x)\,\overline{u(x)}\,v(x)\,\d x,
    \end{equation}
    where
    \begin{equation}\label{e:alessandrini2}
        \left\{
        \begin{array}{rll}
        (-\Delta  +q-\kappa) v &=0,    &\text{ in }   \IB^d,\\
        v  &=f, &\text{ on }   \IS^{d-1},
        \end{array}\right.\qquad    \left\{
        \begin{array}{rll}
        -\Delta u -\overline{\kappa} u &=0,    &\text{ in }   \IB^d,\\
        u  &=g, &\text{ on }   \IS^{d-1}.
        \end{array}\right.     
    \end{equation}
\end{lemma}
\begin{proof}
    First take $f,g\in H^{1/2}(\IS^{d-1})$. Green's formula, when applied to a solution $v$ of \eqref{e:alessandrini2} and any function $u\in H^{1}(\IB^d)$ with $u|_{\IS^{d-1}}=g$, gives:
    \begin{equation*}
        \hp{g}{\Lambda_{q,\kappa} f}_{L^2(\IS^{d-1})} = \int_{\IB^d} (q(x)-\ka)\,\overline{u(x)}\,v(x)\,\d x + \int_{\IB^d} \overline{\nabla u(x)}\,\nabla v(x)\,\d x.
    \end{equation*}
    This yields, as soon as $q$ is real-valued,
    \begin{equation*}
        \hp{g}{\Lambda_{q,\kappa} f}_{L^2(\IS^{d-1})}=\hp{\Lambda_{q,\ol{\kappa}}g}{ f}_{L^2(\IS^{d-1})}.
    \end{equation*}
    Formula \eqref{e:alessandrini1} then follows by particularizing those identities to $q=0$ and subtracting. In addition,  \eqref{e:alessandrini1} shows that $\Lambda_{q,\kappa} - \Lambda_{0,\kappa}$ is bounded on $L^2(\IS^{d-1})$: for some $C=C_{q,d}>0$, 
    \begin{equation*}
        \hp{g}{(\Lambda_{q,\kappa} - \Lambda_{0,\kappa}) f}_{L^2(\IS^{d-1})} \leq \norm{q}_{L^{\infty}(\IB^d)}\norm{u}_{H^{1/2}(\IB^d)}\norm{v}_{H^{1/2}(\IB^d)}\leq C\norm{f}_{L^2(\IS^{d-1}}\norm{g}_{L^2(\IS^{d-1})}
    \end{equation*}
    and therefore \eqref{e:alessandrini1} holds for $f,g\in L^2(\IS^{d-1})$ as claimed.
\end{proof}
Using \Cref{l:sepvar} and \Cref{lemma:alessss}, it is not difficult to compute the Fréchet differential of $\Phi^\ka$ at $q=0$.
\begin{proposition} \label{prop:frechet_der}
Let $q\in L^{\infty}_\rad(\IB^d,\R)$ and $\kappa\in \IC \setminus \Spec_{H^1_0}(-\Delta)$.
The Fréchet differential $\d\Phi^\ka_0$ applied to $q$ is a bounded operator $\d\Phi^\ka_0(q): L^2(\IS^{d-1}) \To L^2(\IS^{d-1})$ invariant under rotations that satisfies:
\begin{equation*}
    \varphi_\ell(\sqrt{\ka})^2\d\Phi^\ka_0 (q)|_{\gH^d_\ell} = \sigma_\ell[q,\ka] \Id_{\gH^d_\ell},   \qquad \forall \ell \in \IN_0.
\end{equation*}
\end{proposition}
\begin{proof}
    From \Cref{lemma:alessss} one can check that
    \begin{equation*}
    \br{g,\d\Phi^\ka_0 (q) f}_{L^2(\IS^{d-1})} 
    = \int_{\IB^d}q (x)\,\overline{u(x)}\,w(x)\,\d x, 
    \end{equation*}
    where $u$ is the same as in \eqref{e:alessandrini2}, and $w$ satisfies that $ (-\Delta  -\kappa) w =0$ on $\IB^d$ and $w|_{\IS^{d-1}} = f$. Note that the previous expression ensures that $\d\Phi^\ka_0(q)$ is a bounded linear operator on $L^2(\IS^{d-1})$.
    Taking $f=g= Y_\ell$ with $Y_\ell \in \H_\ell^d$ normalized in $L^2(\IS^{d-1})$, by \Cref{l:sepvar} and \eqref{eq:conjjj} we find that $w(x) = b_\ell(|x|,\ka)Y_\ell(\hat{x})$ and $\ol{u(x)} = b_\ell(|x|,\ka)\ol{Y_\ell(\hat{x})}$, where $b_\ell$ satisfies \eqref{e:defb}, since we are in the case $q=0$. Therefore, we have
     \begin{equation*}
          \br{Y_\ell,\d\Phi^\ka_0 (q) Y_\ell}_{L^2(\IS^{d-1})} 
          =\int_{\IB^d}q(x)\overline{ b_{\ell}(|x|,\ol{\ka}) }  b_{\ell}(|x|,\ka) |Y_\ell(\hat{x})|^2 \, \d x  = \frac{\sigma_\ell[q,\ka]}{ \varphi_\ell(\sqrt{\ka})^2}  ,
     \end{equation*}
    by \eqref{e:moment}, since $q$ is radial and the integral in the angular variable is $1$. 
\end{proof}

In the next result we show that asymptotically as $\ka \to -\infty$ the spectrum of $\Lambda_{q,\ka}-\Lambda_{0,\ka}$ converges uniformly to the rescaled moments of $q$.
\begin{proposition}\label{prop:spectrum}
    Let $q\in L^\infty_{\rm rad}(\IB^d;\IR)$.  For every $\ka\leq -\norm{q}_{L^\infty(\IB^d)}$ and $\ell\in\N_0$, we have
    \begin{equation*}
    \abs{\lambda_\ell[q,\kappa]-\lambda_\ell[0,\kappa]-\varphi_\ell(\sqrt{\ka})^{-2}\sigma_\ell[q,\kappa]}\leq\frac{\norm{q}^2_{L^\infty(\IB^d)}}{2(\ell+\nu_d)((\ell+\nu_d)^2-\norm{q}_{L^\infty(\IB^d)}-\kappa)} .
\end{equation*}
\end{proposition}
\begin{proof}

Let $Y_\ell \in \H_\ell^d$ be normalized. We apply \eqref{e:alessandrini1}, \eqref{e:alessandrini2} with $f=g=Y_\ell$, so that $\ol{u(x)} = b_\ell(|x|,\ka)\ol{Y_\ell(\hat{x})}$ by \Cref{l:sepvar} and \eqref{eq:conjjj}, since now $\ka= \ol{\ka}$. Thus we have
\begin{align*}
    \lambda_\ell[q,\kappa]-\lambda_\ell[0,\kappa]&=\hp{Y_\ell}{(\Lambda_{q,\kappa}-\Lambda_{0,\kappa})Y_\ell}_{L^2(\Sp^{d-1})}=\int_{\IB^d}q(x) b_\ell(|x|,\ka) \ol{Y_\ell(\widehat{x})}v(x)\,\d x,
\end{align*}
where $v$ satisfies \eqref{e:alessandrini2} with $f=Y_\ell$.
Let $h(x):=v(x)-b_\ell(|x|,\ka) Y_\ell(\widehat{x})$. Inserting this in the previous identity and using that $Y_\ell$ is normalized, yields
\begin{align*}
    \lambda_\ell[q,\kappa]-\lambda_\ell[0,\kappa]
    = \varphi_{\ell}(\sqrt{\kappa})^{-2}\sigma_\ell[q,\kappa]+\int_{\IB^d} q(x)\,b_\ell(|x|,\ka)\cc{Y_\ell(\widehat{x})}h(x)\,\d x.
\end{align*}
The function $h$ satisfies
\begin{equation*}
\begin{cases}
(-\Delta+q(x)-\kappa)h(x)=-q(x)b_\ell(|x|,\ka)Y_\ell(\widehat{x}),&x\in \IB^d,\\
h\vert_{\Sp^{d-1}}=0.
\end{cases}
\end{equation*}
Denoting by $\cR_q(\kappa)$ the Dirichlet resolvent $(-\Delta+q-\ka)^{-1}$, we can write $h=-\cR_q(\kappa)(qg_\kappa)$ with $g_\kappa(x):=b_\ell(|x|,\ka)Y_\ell(\widehat{x})$. Since $-\Delta+q$ and $\cR_q(\kappa)$ leave invariant the subspace $\cE_\ell$ we have
\begin{equation*}
     \|\cR_q(\kappa)|_{\cE_\ell}\|_{\cL(L^2(\IB^d))}\leq\frac{1}{\mathrm{dist}(\kappa, \Spec(-\Delta+q)|_{\cE_\ell})},
\end{equation*}
Using that $qg_{\kappa}\in\cE_\ell$ we can estimate
\begin{align*}
    \abs{\lambda_\ell[q,\kappa]-\lambda_\ell[0,\kappa]-\varphi_{\ell}(\sqrt{\kappa})^{-2}\sigma_\ell[q,\kappa]}&=\vert\hp{qg_{\kappa}}{\cR_q(\kappa)(qg_\kappa)}_{L^2(\IB^d)}\vert\\
    &\leq \frac{\norm{q g_\kappa}_{L^2(\IB^d)}^2}{\mathrm{dist}(\kappa, \Spec(-\Delta+q)|_{\cE_\ell})}\\
    &\leq \frac{ \norm{q}^2_{L^\infty(\IB^d)}\int_{\IB^d}|\varphi_{\ell}(\sqrt{\kappa}|x|)|^2\d x}{|\varphi_{\ell}(\sqrt{\kappa})|^{2}\abs{\IS^{d-1}}\mathrm{dist}(\kappa, \Spec(-\Delta+q)|_{\cE_\ell})}.
\end{align*}
Clearly, $\Spec_{H^1_0}((-\Delta+q)|_{\cE_\ell})\subseteq \Spec_{H^1_0}(-\Delta|_{\cE_\ell})+[-\norm{q}_{L^\infty(\IB^d)},\norm{q}_{L^\infty(\IB^d)}]$ which implies the bound:
\[
\mathrm{dist}(\kappa, \Spec(-\Delta+q)|_{\cE_\ell})\geq j_{\ell+\nu_d,1}^2-\norm{q}_{L^\infty(\IB^d)}-\ka > (\ell+\nu_d)^2-\norm{q}_{L^\infty(\IB^d)}-\ka,
\]
the last estimate following from the inequality $\nu<j_{\nu,1}$ on the first zero of $J_\nu$ for $\nu>0$, \cite[Equation (2.4)]{Laforgia_bessel_zero_low_bound_83}. The result then follows taking \eqref{eq:b_ell_asymp} into account.
\end{proof}

\section{Fourier representation formulas}\label{sec:proofFourier}

In this section we prove an explicit formula to reconstruct a distribution from its moments. The result below can be viewed as an explicit formula for $(\d\Phi^\ka_0)^{-1}|_{\cE_\rad'(\R^d)}$. In particular, it implies that, as soon as $\qb\in\cB_d$ exists, it is uniquely determined by \eqref{eq:moment_prob_formal} in $\cB_d$. 
\begin{theorem}\label{main-thm:Fourier Radial}
Let $f\in \cE_\rad'(\R^d)$ and let $\kappa\in\IC\setminus \{0\}$. Then for $\xi\in\R^d$ we have
\begin{equation}\label{eq:fourierrep}
    \F f(\xi) =
    (2\pi)^d \sum_{\ell=0}^\infty \sigma_\ell[f,\kappa] Z_{\ell,d}\left(1-\frac{\abs{\xi}^2}{2\kappa}\right),
\end{equation}
where the series are absolutely convergent.
\end{theorem}
Note that taking the limit $\ka\to 0$ in \eqref{eq:fourierrep} recovers the representation for $\kappa=0$ that was obtained in \cite{BCMM_2022_Born, Radial_Born}, see \eqref{eq:limit_res}. 
\begin{proof}
We first assume that $f\in L^1_\rad(\IR^d)$ has compact support. Then we can choose $R>0$ such that $f$ is identically 0 outside $R\IB^d$.
Fix $\xi\in \R^d$ and let $\zeta_1,\zeta_2\in\C^d$ be such that 
\begin{equation} \label{eq:z_1_z_2_cond}
    \zeta_1+\zeta_2=-i\xi,\qquad \zeta_1\cdot\zeta_1 =\zeta_2\cdot\zeta_2=-\kappa.
\end{equation}
Writing $f(x)=f_0(|x|)$, we have
\begin{align*}
    \cF f(\xi)&=\int_0^R f_0(r) r^{d-1}\hp{\cc{e_{r\zeta_1}}}{e_{r\zeta_2}}_{L^2(\Sp^{d-1})}\d{r}\\
    &=\int_0^R f_0(r) r^{d-1}\left(\sum_{\ell=0}^\infty\hp{\mathcal{P}_{\ell,d}\cc{e_{r\zeta_1}}}{\mathcal{P}_{\ell,d}e_{r\zeta_2}}_{L^2(\Sp^{d-1})}\right)\d{r}\\
    &=\sum_{\ell=0}^\infty \int_0^R f_0(r) r^{d-1}\hp{\cc{\mathcal{P}_{\ell,d}e_{r\zeta_1}}}{\mathcal{P}_{\ell,d}e_{r\zeta_2}}_{L^2(\Sp^{d-1})}\d{r},
\end{align*}
where $\cP_{\ell,d}$ denotes the orthogonal projection of $L^2(\IS^{d-1})$ onto $\gH_\ell^d$. The exchange of series and integral is justified by Fubini's Theorem and the bound:
\begin{align*}
    \sup_{r\in[0,R]}\sum_{\ell=0}^\infty\abs{\hp{\mathcal{P}_{\ell,d}\cc{e_{r\zeta_1}}}{\mathcal{P}_{\ell,d}e_{r\zeta_2}}_{L^2(\Sp^{d-1})}}\leq \sup_{r\in[0,R]}\norm{e_{r\zeta_1}}_{L^2(\Sp^{d-1})}\norm{e_{r\zeta_2}}_{L^2(\Sp^{d-1})}<\infty.
\end{align*}
Then, \eqref{eq:fourierrep} follows for $\cE_\rad'(\R^d)\cap L^1_\rad(\IB^d)$  from
 $ 2(\zeta_1\cdot\zeta_2)=-(\abs{\xi}^2-2\kappa)$, and the identity
\begin{equation*}
   \hp{\cc{\mathcal{P}_{\ell,d}e_{r\zeta_1}}}{\mathcal{P}_{\ell,d}e_{r\zeta_2}}_{L^2(\Sp^{d-1})} = (2\pi)^d   \varphi_{\ka,\ell}\left(\sqrt{\ka} r\right)^{2}Z_{\ell,d}\left(\frac{\zeta_1\cdot\zeta_2}{\kappa}\right),
\end{equation*}
which is a consequence of  \Cref{prop:mainradial}.

We now prove \eqref{eq:fourierrep} for $f\in \cE_\rad'(\R^d)$ by a density argument. Let $N\in\N_0$ be the order of $f$ as a distribution, and let $R>0$ be such that $\supp f\subseteq R\IB^d$.
Also, choose any $\eta \in \cC^\infty_{c,\rad}(\IR^d,\IR_+)$ with $\int_{\IR^d} \eta (x)\d{x}=1$ and construct the approximate identity $\eta_\varepsilon(x)\coloneqq\epsilon^{-d}\eta(x/\epsilon)$ for $\epsilon\in(0,1)$. Then $f * \eta_\varepsilon\in \cC^\infty_{c,\rad}(\IR^d)$ with 
\[\supp (f* \eta_\varepsilon)\subseteq \supp f+ \epsilon\supp \eta\subseteq \supp f + \epsilon\left(\max_{x\in \supp \eta}\abs{x}\right)\ol{\IB^d},\]
and $\displaystyle\lim_{\epsilon\to0}f* \eta_\varepsilon=f$ in $\cE'(\R^d)$. Therefore, 
\begin{align*}
    \cF f(\xi)&=\br{f,e_{-i\xi}}_{\cE'\times\cC^\infty}=\lim_{\epsilon\to0}\br{f* \eta_\varepsilon,e_{-i\xi}}_{\cE'\times\cC^\infty}=\lim_{\epsilon\to0} \cF (f* \eta_\varepsilon)(\xi)\\
    &= \lim_{\epsilon\to0} (2\pi)^d \sum_{\ell=0}^\infty \sigma_\ell[f* \eta_\varepsilon,\kappa] Z_{\ell,d}\left(1-\frac{\abs{\xi}^2}{2\kappa}\right).
\end{align*}
We wish to apply the dominated convergence theorem to conclude the proof. To this end we bound the terms of the series separately.
For the moments $\sigma_\ell[f* \eta_\varepsilon,\kappa]$ we have
\begin{multline*}
        \abs{\sigma_\ell[f* \eta_\varepsilon,\kappa]}
    =\frac{1}{\abs{\IS^{d-1}}}\abs{\br{f,\eta_\varepsilon* \varphi_{\ka,\ell}^2}_{\cE'\times\cC^\infty}}
    \leq \frac{C_{f,R}}{\abs{\IS^{d-1}}}\norm{\eta_\varepsilon* \varphi_{\ka,\ell}^2}_{\cC^N(R \ol{\IB^d})} \\
    \leq C_{f,R,d}\norm{\varphi_{\ka,\ell}^2}_{\cC^N(R \ol{\IB^d})}.
\end{multline*}
From the definition of $\varphi_\ell$ in \eqref{eq:defvphi} immediately follows that $\IC^d\ni z\longmapsto \varphi_\ell(\sqrt{\kappa (z\cdot z)})^2$ is an entire function bounded by
$\abs{z}^{2\ell}e^{\abs{\kappa}\abs{z}^2}$. Hence, a Cauchy estimate with poly-discs of unit radius gives 
\begin{equation*}
    \norm{\varphi_{\ka,\ell}^2}_{\cC^N(R \ol{\IB^d})}\leq N! \left(R+d\right)^{2\ell}e^{\abs{\kappa}\left(R+d\right)^{2}}=C_{f,R,\kappa,d}(R+d)^{2\ell}.
\end{equation*}
The definition of the Legendre polynomial $P_{\ell,d}$ (see \Cref{app:sh}) implies that the complex phases of all its terms align when we evaluate it at $i\IR$; therefore
\begin{equation*}
    \abs{Z_{\ell,d}(z)}=\frac{N_{\ell,d}}{\abs{\IS^{d-1}}}\abs{P_{\ell,d}(z)}\leq C_d \ell^{d-2}\abs{P_{\ell,d}(i\abs{z})}\leq  C_d \ell^{d-2}(\abs{z}+1)^\ell\abs{P_{\ell,d}(i)},\qquad z\in\IC.
\end{equation*}
To estimate $\abs{P_{\ell,d}(i)}$ we use again its definition from \Cref{app:sh} and the Cauchy integral formula with $\gamma=\{z\in\IC : \abs{z-i}=1\}$ to obtain
\begin{align*}
    \abs{P_{\ell,d}(i)}&=\frac{2^{-\frac{d-3}{2}}R_\ell \ell!}{2\pi}\abs{\int_\gamma\frac{(1-z^2)^{\ell+\frac{d-3}{2}}}{(z-i)^{\ell+1}}\d{z} }\\
    &\leq 2^{-\frac{d-3}{2}}R_\ell \ell!(1+2^2)^{\ell+\frac{d-3}{2}}=\frac{\Gamma(\frac{d-1}{2})\ell!}{\Gamma(\ell+\frac{d-1}{2})}\left(\frac{5}{2}\right)^{\ell+\frac{d-3}{2}}\leq C_d\,\ell\left(\frac{5}{2}\right)^{\ell}.
\end{align*}
Putting all these bounds together, we see that we can indeed apply the dominated convergence theorem in order to exchange limit and series. The proof is then finished by the convergence of $f* \eta_\varepsilon$ to $f$ in $\cE'(\R^d)$.
\end{proof}

\section{Weyl-Titchmarsh functions and the Born approximation}

\subsection{Weyl-Titchmarsh functions and DtN maps}Let us start by providing a description of $\Lambda_{q,\ka}$ that is well-suited for our purposes. We will relate $\Lambda_{q,\ka}$ to the Dirichlet-to-Neumann map of a one-dimensional Schrödinger operator on the half-line:
\begin{equation}   \label{eq:schrodinger_prob}
    \left\{\begin{array}{rll}
    (-\partial_t^2   + Q(t) )  v_z(t) &= -z^2 v_z(t),    &\text{ in }   \R_+,\\
    v_z (0) &= 1.&
    \end{array}\right. 
\end{equation}
When the potential $Q\in L^1_\loc(\IR_+)$ is in the limit-point case, for instance, when
\begin{equation} \label{eq:sharp_L1_condition}
    Q\in A^1(\IR_+)\quad \iff \quad \normm{Q} : =  \sup_{x\in \R_+} \int_{x}^{x+1} |Q(t)| \, \d t < \infty,
\end{equation}
there exists $\beta_Q>0$ such that, when $z\in\IC_+\setminus [0,\beta_Q]$, problem \eqref{eq:schrodinger_prob} has a unique solution $v_z\in L^2(\IR_+)$. One then defines the \textit{Weyl-Titchmarsh function} as:
\begin{equation}\label{eq:m_def}
    m_Q(-z^2) := \partial_t v_z(0).
\end{equation}
The connection of $m_Q$ with radial Dirichlet-to-Neumann maps is as follows. Given a potential $q\in L^p_\rad(\IB^d, \R)$, with $p$ admissible, and $f\in\gH^d_\ell$ then, provided that $\ka \in \IC\setminus \Spec_{H^1_0}(-\Delta+q)$, the unique solution $u\in H^1(\IB^d)$ of
\begin{equation}\label{eq:anotherbvp}
    \left\{\begin{array}{rll}
    (-\Delta  +q-\kappa) u &=0,    &\text{ in }   \IB^d,\\
    u  &=f, &\text{ on }   \IS^{d-1},
    \end{array}\right.   
\end{equation}
can be written in the form
\begin{equation}\label{eq:change_var_radial}
u(x)= |x|^{-\nu_d}v_{\ell+ \nu_d}(-\log|x|)f(\widehat{x}), \quad\text{ with }\hat{x} = x/|x|,
\end{equation}
where $v_z\in L^2(\IR_+)$ solves the boundary value problem \eqref{eq:schrodinger_prob} for the potential $Q$ given by
\begin{equation} \label{eq:q_to_Q}
    Q(t) := e^{-2t}(q_0(e^{-t}) -\ka),\quad q(x)=q_0(|x|),  
    \qquad x\in \R_+.
\end{equation}
It turns out that $m_Q$ completely characterizes $\Lambda_{q,\ka}$.
\begin{lemma} \label{lemma:DtN_mQ}
    Let $p$ be admissible, $q \in L^p_\rad(\IB^d;\R)$ and $ \ka \in \IC \setminus \Spec_{H^1_0}(-\Delta+q)$.
    Then the potential $Q$ given by \eqref{eq:q_to_Q} is in $A^1(\R_+)$, and, if in addition $p>d/2$, then $Q\in L^1(\R_+)$. Therefore, \eqref{eq:schrodinger_prob} has a unique solution that is in $L^2(\IR_+)$, and the spectrum of $\Lambda_{q,\ka}$ can be expressed in terms of the Weyl-Titchmarsh function of $Q$ as
    \begin{equation*}
    \lambda_\ell[q,\kappa] =  -m_Q\left ( -(\ell +\nu_d)^2  \right )-\nu_d,\qquad \ell\in\N_0.
    \end{equation*}
\end{lemma}
\begin{proof}
    That $Q\in A^1(\R_+)$ can be checked by noticing that, since \eqref{eq:p_cond} implies $2p-d\geq 0$
    \begin{equation*}
        \frac{\norm{q-\ka}_{L^p(\IB^d)}^p}{|\IS^{d-1}|} =  \int_0^\infty |Q(t)|^p e^{(2p-d) t} \d t \geq  \int_{0}^\infty |Q(t)|^p \d t \geq \left(\int_x^{x+1} |Q(t)| \d t\right)^p.
    \end{equation*}
    When $2p-d>0$ one can improve this to:
    \begin{equation}\label{ep:estQ}
         \frac{\norm{q-\ka}_{L^p(\IB^d)}^p}{|\IS^{d-1}|}=\int_0^\infty |Q(t)|^p e^{(2p-d) t} \d t \geq \left(\frac{2p-d}{p-1}\right)^{p-1}\|Q\|_{L^1(\IR_+)}^p.  
    \end{equation}
    This ensures that $Q$ satisfies the condition \eqref{eq:sharp_L1_condition}, that \eqref{eq:schrodinger_prob} has a unique square-integrable solution, and that $m_Q$ is well defined. Let $f\in \gH^d_\ell$, and let $v_{\ell+\nu_d}\in L^2(\IR_+)$ solve \eqref{eq:schrodinger_prob}, so that $u$ given by \eqref{eq:change_var_radial} solves \eqref{eq:anotherbvp}. 
    Then $\Lambda_{q,\ka}f=\lambda_\ell[q,\kappa]f$ where
    \begin{equation*}  
    \lambda_\ell[q,\kappa] =  \partial_r \left[r^{-\nu_d}v_\ell(-\log r)\right ]\bigg |_{r=1}   =  -\partial_t v_{\ell+\nu_d}(0) -\nu_d =  -m_Q(-(\ell+\nu_d)^2) -\nu_d,
    \end{equation*} 
    as claimed.
\end{proof}
Let us apply the preceding result to the potential 
\begin{equation} \label{eq:Q_ka_def}
    Q_\ka(t) := -\ka e^{-2t},\qquad \ka \in \IR,
\end{equation}
which corresponds via \eqref{eq:q_to_Q} to the constant potential $-\ka$. When $\ka=0$, it is not difficult to check that its Weyl-Titchmarsh function equals $m_0(-z^2)=-z$. When $\ka\neq 0$, $m_{Q_\ka}$ can be explicitly computed as well.

\begin{lemma} \label{lemma:m_function_ka}
    Suppose $\ka \in \IC \setminus\{0\}$. For every $z\geq 0$ such that $J_z(\sqrt{\ka})\neq 0$,\footnote{This condition is always fulfilled when $\kappa\in\IC\setminus\IR_+$, since all the roots  of $J_z$ are real as soon as $z>-1$.}   the problem 
    \begin{equation*}
        (-\partial_t^2 + Q_\ka)v_z=-z^2v_z,\qquad v_z(0)=1,
    \end{equation*}
    has a unique solution $u_z(\cdot,\ka)\in L^2(\R_+)$. This solution and the Weyl-Titchmarsh function of $Q_\ka$ are, respectively,
    \begin{equation}\label{eq:ukmqk}
        u_z(t,\ka) := \frac{J_z(\sqrt{\ka}e^{-t})}{J_z(\sqrt{\ka})},\qquad m_{Q_\ka}\left ( -z^2  \right ) = -z + \sqrt{\kappa}\dfrac{J_{z+1}(\sqrt{\kappa})}{J_z(\sqrt{\kappa})}.
    \end{equation}
    In particular, for every $\ell\in\IN_0$ such that $\kappa\in\IC\setminus \Spec_{H^1_0}(-\Delta|_{\cE_\ell})$, the radial profile in \eqref{eq:change_var_radial} of the corresponding solution of \eqref{eq:anotherbvp} with $q=0$ and $f\in\gH^d_\ell$ is 
    \begin{equation} \label{eq:radial_profile}
        \frac{1}{r^{\nu_d}}u_z(-\log r,\ka)=  \dfrac{\varphi_{\ell}(\sqrt{\kappa} r)}{\varphi_{\ell}(\sqrt{\kappa})},
    \end{equation}
    and, provided that $\ka\in\IC\setminus\Spec_{H^1_0}(-\Delta)$, the spectrum of $\Lambda_{0,\ka}$ is the sequence
    \begin{equation*}
    \dfrac{\partial_r\varphi_{\ell}(\sqrt{\kappa})}{\varphi_{\ell}(\sqrt{\kappa})}=\ell - \sqrt{\kappa}\dfrac{J_{\ell+1+\nu_d}(\sqrt{\kappa})}{J_{\ell+\nu_d}(\sqrt{\kappa})} ,\qquad\ell\in\IN_0.   
    \end{equation*}
\end{lemma}
\begin{proof}
    The solution $u_z(t,\ka)$ can be explicitly computed as follows. Since $\kappa\neq 0$, we can set $u_z(t, \ka)=w_z(\sqrt{\ka}e^{-t},\ka)$, and direct computation gives that $w_z(\cdot, \ka)$ is a solution of Bessel's equation
    \begin{equation*}
	s^2 w_{z}''(s,\ka) + s w_{z}'(s,\ka) +\left(s^2-z^2\right) w_{z}(s, \ka) = 0,\qquad s\in\IC.
    \end{equation*}
    The condition $u_{z}(0,\kappa)=1$ translates into $w_z(\sqrt{\ka},\ka)=1$; and this forces 
    \[
    w_z(s,\ka)=\frac{J_z(s)}{J_z(\sqrt{\ka})},
    \]
    which gives the claimed formula for $u_z(\cdot,\ka)$.
    This is well defined due to the assumption on $\kappa$.
    The expression  for $m_{Q_\ka}(-z^2)=\partial_t u_z(0,\ka)$ can be obtained using that, for $\nu\geq 0$,
    \begin{equation*}
        sJ_\nu'(s)=\nu J_\nu(s)-sJ_{\nu+1}(s),\qquad \forall s\in\IC.
    \end{equation*}
    Identity \eqref{eq:radial_profile} follows by taking into account that $\Spec_{H^1_0}(-\Delta|_{\cE_\ell})$, the spectrum of the restriction of the Dirichlet Laplacian $-\Delta$ on $\IB^d$ to functions with angular component in $\gH^d_\ell$, is precisely the set of solutions of $\varphi_\ell(\sqrt{\ka}) =0$. Finally, the calculation of the spectrum of $\Lambda_{0,\ka}$ is obtained by \eqref{eq:ukmqk} and the last identity of \Cref{lemma:DtN_mQ}. 
\end{proof}

\subsection{Fixed energy \texorpdfstring{$A_\ka$}{Ak}-amplitude} \label{sec:subsec_1dstate}

For $s\in\R$, we introduce the space $L^1_s(\R_+)$ consisting of those functions $F\in L^1_\loc(\R_+)$ such that 
\begin{equation*}
    \norm{F}_{L^1_s(\R_+)} := \int_0^\infty |F(x)|  e^{-2st} \, \d t.
\end{equation*}
In his seminal paper \cite{Simon_spectral_I}, Simon proved that  for all $Q \in L^1(\R_+)$
there is a $z_Q>0$ and a function $A_Q \in L^1_{z_Q}(\R_+)$ such that
\begin{equation} \label{eq:A_amplitude}
    m_Q(-z^2) -m_{0}(-z^2) = \int_0^\infty A_Q(t) e^{-2zt} \, \d t,
\end{equation}
for all $z$ such that $\Re(z) \ge z_Q$. The function $A_Q$ is known as the $A$-amplitude of the potential $Q$. The class of potentials for which this holds was later widened  to $Q \in L^\infty(\R_+)$ by Gesztesy and Simon in \cite{Simon_spectral_II}. Finally 
Avdonin, Mikhaylov and Rybkin extended this result in \cite{Avdonin_Mikhay_Rybkin_07} for potentials in $L^1_\loc(\R_+)$ such that \eqref{eq:sharp_L1_condition} holds. 
In particular, they proved the existence  of $A_Q \in L^1_{z_Q}(\R_+)$ with
\begin{equation} \label{eq:z_q_choice}
    z_Q :=  2e(\normm{Q} +\sqrt{\normm{Q}}),
\end{equation}
such that \eqref{eq:A_amplitude} holds.

We will generalize the notion of $A$-amplitude to obtain a similar representation formula to  \eqref{eq:A_amplitude} when $m_0$ is replaced by $m_{Q_\ka}$ for the same class of potentials used in \cite{Avdonin_Mikhay_Rybkin_07}. This will lead to the notion of $A_\ka$-amplitude, whose existence we prove and analyze some of its properties (local uniqueness and stability, among others). This will allow us to prove the existence of the Born approximation for the fixed-energy Calderón problem and obtain many of its properties.

Let us start by analyzing the Weyl-Titchmarsh function of the potentials $Q_\ka$ defined in \eqref{eq:Q_ka_def}. 

\begin{lemma} \label{lemma:m_function}
    For every $\ka \in \R \setminus\{0\}$ there exists a function $A_{Q_\ka} \in L^1_{z_\ka}(\R_+)$ such that 
    \begin{equation*}
        m_{Q_\ka}(-z^2) -m_{0}(-z^2) = \int_0^\infty A_{Q_\ka}(t) e^{-2zt} \, \d t,
    \end{equation*}
    for all $z\in \IC_+$ satisfying  $\Re(z) \ge z_\ka$, where $z_\ka$ is any real number such that
    \begin{equation}\label{eq:zk}
        z_\ka \ge 0, \qquad z_\ka > \frac{\ka}{\sqrt{|\ka|}} -j_{0,1}.
    \end{equation}
\end{lemma}
\begin{proof}
    This follows from the fact that   $J_z(\sqrt{\ka})$  has no non-negative zeros for $\ka$ fixed if $\ka <0 $. If $\ka>0$, then  the largest zero $z_L$ satisfies $ j_{z_L,1} = \sqrt{\ka}$, and there are no zeros when $z$ is in the region determined by the condition  $ j_{z,1} > \sqrt{\ka}$. Since  from \cite[Equation (2.4)]{Laforgia_bessel_zero_low_bound_83} we have the lower  bound $j_{z,1} \ge  z + j_{0,1}$ when $z\ge 0$, it follows that $J_z(\sqrt{\ka})$ has no zeros if $ z + j_{0,1} > \sqrt{\ka} $ and $z\ge 0$. In other words, non-negative zeros are contained in the interval $[0,\sqrt{\ka}-j_{0,1}]$ when $\sqrt{\ka} \ge j_{0,1}$,  and there are no zeros in $[0,\infty)$ if $\sqrt{\ka} < j_{0,1}$. By \cite[Theorem 1.3]{Simon_spectral_III} we conclude the existence of $A_{Q_\ka}\in L^1_{z_\ka}(\R_+)$ where $z_\ka$ is defined in the statement.
\end{proof}
Define
\begin{equation} \label{eq:alpha_q_def}
    \alpha_Q(t) := \int_0^t |Q(s)| \, \d s.
\end{equation}

The main properties of the $A_\ka$-amplitude that are proved in this article are contained in the following two theorems.
\begin{theorem}  \label{thm:A-amplitude}
    Let $Q \in A^1(\R_+)$  and $\ka \in \R$.
    There exists a unique function  $A_{Q,\ka} \in L^1_{s}(\R_+)$, where $s:= \max (z_Q, z_\ka )$, such that
    \begin{equation*}
        m_Q(-z^2) -m_{Q_\ka}(-z^2) = \int_0^\infty A_{Q,\ka}(t) u_z(t,\ka)^2 \, \d t,
    \end{equation*}
    for all $z\in \IC$ with $\Re(z) \ge s$, with $u_z(\cdot,\ka)$ given by \eqref{eq:ukmqk}. The function $A_{Q,0}$ coincides with $A_Q$, the usual $A$-amplitude of $Q$. Moreover, 
    \begin{enumerate}[i)]
        \item $A_{Q,\ka} - Q\in \cC(\IR_+)$ and $\lim_{t\to 0^+}(A_{Q,\ka} - Q)(t)=0$.
        \item There exists $C_\ka>0$ such that, for every $x\in\IR_+$, the following estimate holds:
            \begin{equation} \label{eq:est_thm_1d_2}
            |A_{Q,\ka}(t) + Q_\ka(t) - Q(t)| \le   C_\ka(1+\alpha_Q(t)) e^{ t(\alpha_{Q}(t) +|\ka|) } \int_{0}^t |Q(s)-Q_\ka(s)|\, \d s .
            \end{equation}
        \item If $Q \in \cC^{m}([0,\infty))$ for some $m\in \N_0$, then $A_{Q,\ka} - Q$ belongs to $\cC^{m+1}([0,\infty))$.
        \item If $Q\ge 0$ $a.e.$ in $\R_+$, $Q(1+|\cdot|) \in L^1(\R_+)$, and $\ka < (j_{0,1})^2$, then  $A_{Q,\ka} \in L^1(\R_+)$.
    \end{enumerate}
\end{theorem}
The function $A_{Q,\ka}$ will be called the $A_\ka$\text{-amplitude of} $Q$. We also prove that the correspondence that maps a potential to its $A_\ka$-amplitude is injective, and that its inverse is locally Hölder continuous.

 \begin{theorem}  \label{thm:stability_A_function}
    For every $a>0$ the following statements hold.
    \begin{enumerate}[i)]
        \item Let $Q_1,Q_2 \in A^1(\R_+)$. 
        \begin{equation*}
        A_{Q_1,\ka}|_{(0,a)} = A_{Q_2,\ka}|_{(0,a)}   \; \iff \; Q_1|_{(0,a)} = Q_2|_{(0,a)} .
        \end{equation*}
        \item For every $1<p\le \infty$ and $M>0$ there exist $\varepsilon(\ka,a,p), C_1(\ka,a,p,M)>0$ such that for every $Q_1,Q_2 \in A^1(\IR_+)$ satisfying
        \begin{equation} \label{est:A_condition}
        \norm{Q_j}_{L^p((0,a))} <  M, \; j=1,2,\qquad \int_0^a  \left| A_{Q_1,\ka}(t) - A_{Q_2,\ka}(t) \right| \, \d t <\varepsilon(\ka,a,p),
        \end{equation}
        one has 
        \begin{equation*}
        \int_0^{a}  \left| Q_1(t)-Q_2(t) \right| \,\d t    < C_1(\ka,a,p,M)
        \left(\int_0^a  \left| A_{Q_1,\ka}(t) - A_{Q_2,\ka}(t) \right| \, \d t\right)^{1/(p'+1)}.
        \end{equation*}
    \end{enumerate} 
 \end{theorem}
 The proofs of these results will be presented in \Cref{sec:subsec_1d}.

\subsection{From the \texorpdfstring{$A_\ka$}{Ak}-amplitude to the Born approximation} \

Some rather straightforward consequences of \Cref{thm:A-amplitude} in the context of the Calderón problem are gathered in the next result. Recall our convention:
\begin{equation} \label{eq:q_to_Q_log}
    q(x) -\ka = |x|^{-2}Q(-\log |x|), \qquad x\in \R_+.
\end{equation}
\begin{proposition} \label{prop:q_s_def}
     Let $ \ka \in \IR$ and $p$ be admissible. Take $q\in X_{p,\ka}(\IB^d)$, and let $Q$ be given by \eqref{eq:q_to_Q}. Denote by $A_{Q,\ka}$ the $A_\ka$-amplitude of $Q$ and write
        \begin{equation} \label{eq:qs_def}
        \qs(x) : = 
        \frac{A_{Q,\ka}(-\log |x|)}{|x|^2}, \qquad  x\in \IB^d\setminus\{0\}.
    \end{equation}
    The following statements hold.
    \begin{enumerate}[i)]
        \item $\qs \in L^1_{\ell_q}(\IB^d)$, 
    where $\ell_q := \max \left\{0,\max (z_Q, z_\ka ) -\nu_d \right\}$ ($z_\ka$ and $z_Q$  were defined in \eqref{eq:zk} and \eqref{eq:z_q_choice}, respectively, and the space $ L^1_{\ell_q}(\IB^d)$ in \eqref{eq:norm_L1_wheight}).
        \item For all $\ell \in \N_0$ with $\ell \ge \ell_q$,
        \begin{equation} \label{eq:lambda_q_qs}
        \sigma_\ell[\qs,\kappa]  =  \frac{1}{\abs{\Sp^{d-1}}}\int_{\IB^d} \qs(x) \varphi_{\ell}(\sqrt{\kappa}\abs{x})^2   \, \d x = (\lambda_\ell[q,\kappa]  -\lambda_\ell[0,\kappa])\varphi_{\ell}(\sqrt{\kappa})^2  .
        \end{equation}
        \item $\qs - q\in \cC(\ol{\IB^d}\setminus\{0\})$ and $\qs-q=0$ in $\partial \IB^d$. If $q \in \cC^{m}(\ol{\IB^d} \setminus \{0\})$  for some $m\in \N_0$ then $\qs - q$ belongs to $\cC^{m+1}(\ol{\IB^d} \setminus \{0\})$.
        \item Let $p>d/2$. There exist constants $C_{q,\ka}>0$ and $\beta_{q,\ka}>0$  such that
    \begin{equation} \label{eq:pol_sing}
        |\qs(x)-q(x)| \le C_{q,\ka}\frac{1}{|x|^{\beta_{q,\ka}}} \int_{|x|<|y|<1} |q(y)| \, \d y.
    \end{equation}
    \item If $p>d/2$, $\ka\le  (j_{0,1})^2$, and $q-\ka \ge 0$  $a.e.$ in $\IB^d$, then $ \qs \in L^1_\rad(\IB^d)$.
    \end{enumerate}
\end{proposition}
\begin{proof} 
    To check (i) simply note that $A_{Q,\ka} \in L^1_{s}(\R_+)$ with   $s= \max(z_Q,z_\ka)$, then $\qs \in L^1_{\ell_q}(\IB^d)$, by \Cref{thm:A-amplitude}. Changing  variables $t=-\log r$ gives:
    \begin{equation} \label{eq:L1_1d_ball}
        \int_{\R_+}    |A_{Q,\ka}(t)|  e^{-2st} \, \d t = \frac{1}{\abs{\Sp^{d-1}}}\int_{\IB^d} |\qs(x)| |x|^{2(s-\nu_d)} \, \d x.
    \end{equation}
    Let us prove (ii). 
    Assume first that, in addition, $\ka \in \IR \setminus \Spec_{H^1_0}(-\Delta)$.  \Cref{lemma:m_function_ka} ensures that
    \begin{equation*} 
    e^{-\nu_d t} \varphi_\ell(\sqrt{\ka}e^{-t})=u_{\ell+\nu_d}(t,\ka)\varphi_\ell(\sqrt{\ka}).
    \end{equation*}
    If $Q$ is given by \eqref{eq:q_to_Q} and $\ell \ge \ell_q$ then, by \Cref{lemma:DtN_mQ} and \Cref{thm:A-amplitude},
    \begin{multline*}
        \lambda_\ell[q,\kappa]  -\lambda_\ell[0,\kappa]  
        = -m_Q\left ( -(\ell +\nu_d)^2  \right ) + m_{Q_\ka}\left ( -(\ell +\nu_d)^2  \right )  \\
        = -\int_0^\infty A_{Q,\ka}(t) u_{\ell+\nu_d}(t,\ka)^2 \, \d t 
        = -\int_0^\infty A_{Q,\ka}(t)  e^{-2\nu_d t} \frac{\varphi_\ell(\sqrt{\ka}e^{-t})^2}{\varphi_\ell(\sqrt{\ka})^2} \, \d t.
    \end{multline*}
    Changing variables via \eqref{eq:change_var_radial} and $t= -\log r$ we obtain
    \begin{equation*}
        \lambda_\ell[q,\kappa]  -\lambda_\ell[0,\kappa]  
        = \int_0^1 \frac{A_{Q,\ka}(-\log r)}{r^2}  \frac{\varphi_\ell(\sqrt{\ka}r)^2}{\varphi_\ell(\sqrt{\ka})^2}  r^{d-1} \, \d r ,
    \end{equation*}
    which yields
    \begin{equation*}
       \left( \lambda_\ell[q,\kappa]  -\lambda_\ell[0,\kappa] \right) \varphi_\ell(\sqrt{\ka})^2 = \frac{1}{\abs{\Sp^{d-1}}}\int_{\IB^d} \frac{A_{Q,\ka}(-\log |x|)}{|x|^2} \varphi_\ell(\sqrt{\ka}|x|)^2   \, \d x, \qquad \forall\ell \ge \ell_q,
    \end{equation*} 
    which proves \eqref{eq:lambda_q_qs} when $\ka \notin\Spec_{H^1_0}(-\Delta)$. To see that the result holds only under the assumption that $\ka$ is not a Dirichlet eigenvalue of $-\Delta+q$  notice that both sides of the above equations can be extended continuously  to $\Spec_{H^1_0}(-\Delta)$  by \eqref{eq:defvphi} and \Cref{lemma:m_function_ka}.

    Property (iii) follows from \Cref{thm:A-amplitude} (i) and (iii):  $A_{Q,\ka}-Q$ is continuous on $\R_+$ and can be extended continuously by zero to $[0,\infty)$.  Then $\qs -q$ is continuous on $\IB^d\setminus\{0\}$, and can be continuously extended by zero to $\ol{\IB^d}\setminus\{0\}$. 
    
    Part (iv) follows from estimate \eqref{eq:est_thm_1d_2} by \eqref{eq:q_to_Q_log} and a change of variables. To see this, recall that  $Q\in L^1(\R_+)$ when $p>d/2$ (see \Cref{lemma:DtN_mQ}), then $\alpha_Q(t) \le \norm{Q}_{L^1(\R_+)}$ for all $t\in \R$. 
    Hence, there exists $\beta_{q,\ka}>0$ such that $|x|^{ -(2\alpha_{Q}(-\log|x|) +|\ka|) } \le |x|^{-\beta_{q,\ka}} $.

    We now prove (v). Using a similar estimate to \eqref{ep:estQ}, one can verify that, if $q \in L_\rad^p(\IB^d;\R)$ with $p>d/2$, then $Q(1+|\cdot|)^\alpha \in L^1(\R_+)$ for all $\alpha \ge 0$.
    Also if $q-\ka \ge 0$ and $\ka <(j_{0,1})^2$, then $Q \ge 0$ $a.e.$ in $\R_+$. Thus, by \Cref{thm:A-amplitude} (iv) we obtain that $A_{Q,\ka} \in L^1(\R_+)$. 
    Using \eqref{eq:L1_1d_ball} with $s=0$, it follows that $ \qs \in L^1_\rad(\IB^d)$.
\end{proof}

\section{Proof of the main theorems}\label{sec:proofs}

\subsection{Proof of \texorpdfstring{\Cref{main-thm:existence_Born}}{Theorem 1} and   \texorpdfstring{\Cref{main-thm:basic_properties}}{Theorem 2}}

Consider the function $\qs$ defined in \Cref{prop:q_s_def}.  
Since $\qs\in L^1_{\ell_0}(\IB^d)$, this function might have an algebraic singularity at the origin. Here we identify $\qs$ with its extension by zero to all $\R^d$.
    
Let $F\in\cE'_\rad(\R^d)$ be a regularization of the function $\qs$ in the sense of \cite[Proposition 1 p. 11]{GelShilVol1} (since $\qs$ is radial, $F$ can also be chosen to be radial). By definition, the distribution $F$ satisfies
\begin{equation} \label{eq:reg}
    \br{F,\phi}_{\cE'\times\cC^\infty} = \br{\qs, \phi}_{\cE'\times\cC^\infty}, \qquad \text{for all } \, \phi \in \cC_c^\infty(\R^d\setminus\{0\}).
\end{equation}
It turns out that there exists $N\in \N$, such that  $N\ge \ell_q$ and\footnote{Note that, although the moments in the left-hand side must be understood in distributional sense (see \eqref{e:moment_distributions}), the ones in the right-hand side are given by \eqref{eq:lambda_q_qs}.}
\begin{equation} \label{eq:moment_identity}
    \sigma_\ell[F,\ka] = \sigma_\ell[\qs,\ka] , \qquad \text{for all } \ell \ge N.
\end{equation}
To see why \eqref{eq:moment_identity} holds, write $w_\alpha(x):=|x|^\alpha$ and note that $w_{2\ell_q}\qs  \in L^1(\IB^d)$, so in particular $w_{2\ell_q}\qs\in \cE_\rad'(\R^d)$. The identity \eqref{eq:reg} implies that
\begin{equation*} 
    w_{2\ell_q}F - w_{2\ell_q}\qs  = \mu\in \cE_\rad'(\R^d),
\end{equation*}
and $\mu$ is a distribution supported in $\{0\}$. Since $\mu$ must be a finite linear combination of Dirac deltas and its derivatives, there exists $N\ge \ell_q$ such that $ w_{2\ell}F  = w_{2\ell}\qs$ for all $\ell \ge N$. Since ${\varphi_{\ka,\ell}}(|x|)^2 \sim |\sqrt{\ka}x|^{2\ell}$ as $|x|\to 0$ by \eqref{eq:defvphi}, we conclude that
\begin{equation*} 
    \varphi_{\ka,\ell}^2F  = \varphi_{\ka,\ell}^2\qs,  \qquad \text{for all }\ell\ge N,
\end{equation*}
in the sense of distributions. This proves \eqref{eq:moment_identity}.

As a consequence of \eqref{eq:moment_identity} and \eqref{eq:lambda_q_qs} we have
\begin{equation*}
    \sigma_\ell[F,\ka] = (\lambda_\ell[q,\kappa] - \lambda_\ell[0,\kappa])\varphi_{\ell}(\sqrt{\ka})^2, \qquad \text{for all } \ell \ge N.
\end{equation*}
Therefore, by \Cref{main-thm:Fourier Radial}  we must have that
\begin{equation*}                     
    \F (F)(\xi) = (2\pi)^d \sum_{\ell=0}^\infty (\lambda_\ell[q,\kappa] - \lambda_\ell[0,\kappa])\varphi_{\ell}(\sqrt{\ka})^2  Z_{\ell,d}\left(1-\frac{\abs{\xi}^2}{2\kappa}\right)  +P_N(|\xi|^2),
\end{equation*}
where $P_N$ is a polynomial. This identity implies that  the first term in the right-hand side is the Fourier transform of a tempered distribution that we denote $\qb$, so that  we get the identity $ \F (\qb)(\xi) = \F (F)(\xi)-P_N(|\xi|^2)$. Since $P_N(|\xi|^2)$ is the Fourier transform of a distribution supported in the origin, $\qb$ is also supported in $\ol{\IB^d}$. Also, $\qb$ coincides with $\qs$ outside the origin:
\begin{equation}\label{eq:qs_qb_identity}
        \qb|_{\IB^d\setminus \{0\}} = \qs,
\end{equation}
 which implies that $\qb \in \cB_d$ (here we are using the identification of elements of $\cB_d$ extended by zero, with elements of $\cE'_\rad(\R^d)$).
Finally, again by \Cref{main-thm:Fourier Radial}, $\qb$ satisfies \eqref{eq:Born_moments}. Note that once the existence of the Born approximation at energy $\ka$ has been established, \Cref{main-thm:basic_properties} follows from \Cref{prop:q_s_def} and \eqref{eq:qs_qb_identity}. 

The proof of \Cref{main-thm:existence_Born} will be concluded once we prove identity \eqref{eq:linear_cgos}.
The conditions on $\ka$ and $q$ in the statement guarantee that the operator $\Lambda_{\kappa,q}-\Lambda_{\kappa,0}$ is well defined.
Recall that $\cP_{\ell,d}$ denotes the orthogonal projection of $L^2(\IS^{d-1})$ onto $\gH_\ell^d$, and that $\cP_{\ell,d} \ol{f} = \ol{\cP_{\ell,d}f} $. Therefore we have that
    \begin{equation} \label{eq:pre_series} 
       \nonumber \hp{\cc{e_{\zeta_1}}}{(\Lambda_{\kappa,q}-\Lambda_{\kappa,0})e_{\zeta_2}}_{L^2(\Sp^{d-1})} 
       =\sum_{\ell=0}^\infty(\lambda_\ell[q,\kappa]-\lambda_\ell[0,\kappa])\hp{\cc{\mathcal{P}_{\ell,d}{e_{\zeta_1}}}}{\mathcal{P}_{\ell,d}e_{\zeta_2}}_{L^2(\Sp^{d-1})}.
    \end{equation}
    The absolute convergence of this series is immediate from the boundedness in $L^2(\IS^{d-1})$ of   $\Lambda_{\kappa,q}-\Lambda_{\kappa,0}$ and the estimate
    \begin{align*}\sum_{\ell=0}^\infty\abs{\hp{\mathcal{P}_{\ell,d}\cc{e_{\zeta_1}}}{\mathcal{P}_{\ell,d}e_{\zeta_2}}_{L^2(\Sp^{d-1})}}\leq \norm{e_{\zeta_1}}_{L^2(\Sp^{d-1})}\norm{e_{\zeta_2}}_{L^2(\Sp^{d-1})}<\infty.
    \end{align*}
    Since $2(\zeta_1\cdot\zeta_2)=-(\abs{\xi}^2-2\kappa)$,
    \Cref{prop:mainradial} shows that the series in the right-hand side of \eqref{eq:pre_series} coincides exactly  with the series given in \eqref{eq:Born_Fourier_introduction} which concludes the proof of \eqref{eq:linear_cgos}.

\subsection{Proof of \texorpdfstring{\Cref{main-thm:stability}}{Theorem 3}}

    Recall $U_b:= \{x\in\IB^d: b<|x|< 1\}$. By  \eqref{eq:qs_def}, \eqref{eq:q_to_Q_log} and \eqref{eq:qs_qb_identity} we have
    \begin{equation*}
        q_j(x) -\ka = |x|^{-2}Q_j(-\log |x|),\qquad q_{j,\ka}^\mB(x)  = 
        |x|^{-2} A_{Q_j,\ka}(-\log |x|), 
    \qquad x\in \IB^d.
    \end{equation*}
    Then the uniqueness  result follows immediately from \Cref{thm:stability_A_function}(i) with $a = -\log b$.
    
    For the stability notice that if $\max_{j=1,2} \norm{q_j}_{L^p(U_b)}  <  K$, then one can choose a constant $M>0$ (dependent on $\ka$) such that \eqref{est:A_condition}. Applying \Cref{thm:stability_A_function}(ii), and letting $a = -\log b$, we obtain
    \begin{equation*}
        \int_{b<|x|<1} | q_1(x)- q_2(x)| |x|^{2-d}  \,  \d x =    \int_0^a | Q_1(t)- Q_2(t)| \, \d t,
    \end{equation*}
    as well as the analogous identity for the difference of the Born approximations and the $A_\ka$ amplitudes. Bounding above and below the weight $|x|^{2-d}$ finishes the proof of the theorem.

\subsection{Proof of  \texorpdfstring{\Cref{main-thm:limit}}{Theorem 4}}
    We assume from the start that $\kappa<-\norm{q}_{L^\infty(\IB^d)}$, so that 
    \[q(x)-\kappa=q(x)+\abs{\kappa}>0
    ,\qquad x\in\IB^d.\]
    Let $\zeta_1,\zeta_2\in \C^d$ satisfy \eqref{eq:z_1_z_2_cond}. Then $(-\Delta-\kappa)e_{\zeta_i}=0$ for $i=1,2$ and by Alessandrini's identity \eqref{e:alessandrini1} we have
    \begin{equation*}
        \hp{\cc{e_{\zeta_1}}}{(\Lambda_{q,\kappa}-\Lambda_{0,\kappa})e_{\zeta_2}}_{L^2(\Sp^{d-1})}=\int_{\IB^d} q(x)\,e_{\zeta_1}(x)\,v(x)\,\d x,
    \end{equation*}    
    where $v$ satisfies the equation on the right of \eqref{e:alessandrini2} with $f= e_{\zeta_2}$.
    Notice that $v-e_{\zeta_2}$ satisfies the equation
    \begin{equation*}
            \left\{
    \begin{array}{rll}
    (-\Delta  + q-\kappa) (v-e_{\zeta_2}) &=-q e_{\zeta_2}   &\text{ in }   \IB^d,\\
    v  &=0, &\text{ on }   \IS^{d-1},
    \end{array}\right.
    \end{equation*}
    so can rewrite the previous equality as
    \begin{equation*}
        \hp{\cc{e_{\zeta_1}}}{(\Lambda_{q,\kappa}-\Lambda_{0,\kappa})e_{\zeta_2}}_{L^2(\Sp^{d-1})}-\F q(\xi)=-\hp{q\cc{e_{\zeta_1}}}{\cR_q(\kappa)(q e_{\zeta_2})}_{L^2(\IB^{d})}.
    \end{equation*}
    Since $q-\kappa>0$, we have that $\cR_q(\kappa)(q e_{\zeta_2})=\int_0^\infty e^{- t(-\Delta  + q-\kappa)}(q e_{\zeta_2})\d{t}$. Writing the semigroup via the Feynman–Kac formula \cite[Section 1.3]{Sznitman} and exchanging integral and expectation (which is justified by Fubini and $\norm{qe_{\zeta_2}}_{L^\infty(\IB^d)}<\infty$) we obtain
    \[\cR_q(\kappa)(q e_{\zeta_2})(x)=\frac{1}{2}\IE_x\left[\int_0^{\tau_{\IB^d}}(q e_{\zeta_2})(W_t)\exp\left(-\int_0^t \frac{q(W_s)+\abs{\kappa}}{2}\d s\right)\d t\right],\]
    where $W_t$ is a standard $d$-dimensional Brownian motion and $\tau_{\IB^d}$ is the exit time of $\IB^d$. Therefore
    \begin{align*}
            \abs{\cR_q(\kappa)(q e_{\zeta_2})(x)}&\leq
        \frac{\norm{q}_{L^\infty(\IB^d)}}{2}\IE_x\left[\int_0^{\tau_{\IB^d}}\abs{e_{\zeta_2}(W_t)}\exp\left(-\frac{t}{2}\left(\abs{\kappa}-\norm{q}_{L^\infty(\IB^d)}\right)\right)\d t\right]\\
        &\leq \frac{\norm{q}_{L^\infty(\IB^d)}}{2}\IE_x\left[\int_0^{\infty}(\chi_{\IB^d}e_{\Re\zeta_2})(W_t)\exp\left(-\frac{t}{2}\left(\abs{\kappa}-\norm{q}_{L^\infty(\IB^d)}\right)\right)\d t\right]\\
        &=\norm{q}_{L^\infty(\IB^d)}\cR^{\IR^d}_{0}\left(-\abs{\kappa}+\norm{q}_{L^\infty(\IB^d)}\right)(\chi_{\IB^d}e_{\Re\zeta_2})(x),
    \end{align*}
    where $\cR^{\IR^d}_{0}(-z^2)$ is the inverse of $-\Delta+z^2$ defined in $L^2(\R^d)$. This operator in the whole space has the Green function
    \[G_{-z^2}(r)\coloneqq(2\pi)^{-\frac{d}{2}}\left(\frac{z}{r}\right)^{\nu_d}K_{\nu_d}\left(z r\right),\qquad z>0,\]
    where $K_\nu$ is the modified Bessel function of order $\nu$.
    Hence, with $z=\sqrt{\abs{\kappa}-\norm{q}_{L^\infty(\IB^d)}}>0$, we have
    \begin{align*}
    \abs{\hp{q\cc{e_{\zeta_1}}}{\cR_q(\kappa)(q e_{\zeta_2})}_{L^2(\IB^{d})}} &\leq \hp{\abs{q\cc{e_{\zeta_1}}}}{\abs{\cR_q(\kappa)(q e_{\zeta_2})}}_{L^2(\IB^{d})}\\
    &=\hp{\abs{q}e_{-\Re\zeta_2}}{\abs{\cR_q(\kappa)(q e_{\zeta_2})}}_{L^2(\IB^{d})}\\
    &\leq \norm{q}^2_{L^\infty(\IB^d)}\int_{\IB^d}e_{-\Re\zeta_2}(x)\int_{\IB^d}G_{-z^2}(\abs{y-x})e_{\Re\zeta_2}(y)\d y \d x\\
    &=\norm{q}^2_{L^\infty(\IB^d)}\int_{\IB^d}\int_{-x+\IB^d}G_{-z^2}(\abs{y})e_{\Re\zeta_2}(y)\d y \d x\\
    &\leq\norm{q}^2_{L^\infty(\IB^d)}|\IB^d|\int_{2\IB^d}G_{-z^2}(\abs{y})e_{\Re\zeta_2}(y)\d y.
    \end{align*}
    Once again, the conditions \eqref{eq:z_1_z_2_cond} on $\zeta_1,\zeta_2$ imply that they must be of the form 
    \[\zeta_{\genfrac{}{}{0pt}{3}{1}{2}}=-\frac{i}{2}\xi\mp\sqrt{-\kappa+\frac{\abs{\xi}^2}{4}}\widehat{\xi}_\perp=-\frac{i}{2}\xi\mp\sqrt{\abs{\kappa}+\frac{\abs{\xi}^2}{4}}\widehat{\xi}_\perp,\]
    where $\widehat{\xi}_\perp\in\R^d$ is any unit vector perpendicular to $\xi$. Since our last integral is invariant under rotations of $\Re\zeta_2=\sqrt{\abs{\kappa}+\frac{\abs{\xi}^2}{4}}\widehat{\xi}_\perp$, we can choose $\widehat{\xi}_\perp$ to be $\widehat{x}_1$, so that when using polar coordinates we obtain
    \begin{align*}
        \int_{2\IB^d}G_{-z^2}&(\abs{y})e_{\Re\zeta_2}(y)\d y\\
        &=\frac{2\pi^\frac{d-1}{2}}{\Gamma(\frac{d-1}{2})}\int_0^2\int_0^\pi G_{-z^2}(r)e^{\abs{\Re\zeta_2}r\cos\theta}(\sin\theta)^{d-2}r^{d-1}\d\theta\d r\\
        &=(2\pi)^\frac{d}{2}\int_0^2 \frac{G_{-z^2}(r)I_{\nu_d}(\abs{\Re\zeta_2}r)r^{d-1}}{(\abs{\Re\zeta_2}r)^{\nu_d}}\d r\\
        &=\left(\frac{z}{\abs{\Re\zeta_2}}\right)^{\nu_d}\int_0^2 K_{\nu_d}(z r)I_{\nu_d}(\abs{\Re\zeta_2}r)r \d r\\
        &=\frac{1}{\abs{\Re\zeta_2}^2-z^2}\left(\frac{z}{\abs{\Re\zeta_2}}\right)^{\nu_d}\\
        &\qquad\times\left( -\left(\frac{\abs{\Re\zeta_2}}{z}\right)^{\nu_d}+2\abs{\Re\zeta_2}K_{\nu_d}(2z)I_{\nu_d+1}(2\abs{\Re\zeta_2})+2 z K_{\nu_d+1}(2z)I_{\nu_d}(2\abs{\Re\zeta_2})\right).
    \end{align*}
    The proof is finished by the asymptotics \cite[Section 7.23]{Watson1944}
    \[K_\nu(x)=\sqrt{\frac{\pi}{2x}}e^{-x}(1+o(1)),\qquad I_\nu(x)=\sqrt{\frac{1}{2\pi x}}e^{x}(1+o(1)),\qquad x\to\infty.\]

\section{Proof of  \texorpdfstring{\Cref{thm:A-amplitude}}{Theorem 4.4} and \texorpdfstring{\Cref{thm:stability_A_function}}{Theorem 4.5}   }\label{sec:subsec_1d}

In this section we present the proofs of our two main auxiliary results.

\begin{proof}[Proof of \Cref{thm:A-amplitude}]
    Start by noticing that, by \eqref{eq:A_amplitude} and \Cref{lemma:m_function}, for all $z \ge s =\max\left(z_Q, z_\ka \right)$ one has
    \begin{align*}
        m_Q(-z^2) -m_{Q_\ka}(-z^2) &= m_Q(-z^2) -m_0(-z^2) -(m_{Q_\ka}(-z^2)-m_0(-z^2)) 
        \\ &= \int_{0}^{\infty}  \left (A_Q(t)-A_{Q_\ka}(t)\right)  e^{-2zt} \, \d t,
    \end{align*}
    where $A_Q-A_{Q_\ka}$ belongs to $L^1_s(\R_+)$. Let  $\psi_z(\cdot;\ka)\in L^2(\IR_+)$ denote the Jost solution of \eqref{eq:uz_def_b}. This solution is characterized by identities \eqref{eq:jost_asymptotic} and \eqref{eq:id_jost_sol}. Using  \Cref{lemma:Change_basis_3} iv) we can write
    \begin{equation*}
         m_Q(-z^2) -m_{Q_\ka}(-z^2) = \frac{1}{\psi_z(0;\ka)^2}  \int_{0}^{\infty} \cT_{\ka} \left (A_Q-A_{Q_\ka}\right) (t) e^{-2zt} \, \d t,
    \end{equation*}
    for all $z \ge s$. Then, by \Cref{lemma:Change_basis} iv), we have
     \begin{equation*}
         m_Q(-z^2) -m_{Q_\ka}(-z^2) =\int_{0}^{\infty}  \cG_\ka^{-1} \cT_\ka \left (A_Q-A_{Q_\ka}\right) (t)  \frac{\psi_z(t;\ka)^2}{\psi_z(0;\ka)^2}  \, \d t,
    \end{equation*}
    for all $z \ge s$. This identity can be analytically extended for all $z\in\IC$ such that  $\Re(z) \ge s$,
    which proves the first statement of the theorem with
    \begin{equation} \label{eq:A_Q}
        A_{Q,\ka} :=  \cG_\ka^{-1} \cT_\ka \left (A_Q-A_{Q_\ka}\right).
    \end{equation}

    On the one hand by  \cite{Simon_spectral_I} we have always that $A_Q-Q \in \cC(\R_+)$,  and that   $A_Q-Q \in \cC^{m+2}(\R_+)$ if $Q\in \cC^{m}(\R_+)$ for all $m\in \N_0$. 
    \\
    On the other hand, by \Cref{lemma:Change_basis} and \Cref{lemma:Change_basis_3},
    $\cG_\ka^{-1} -\Id$ and $\cT_\ka-\Id$ are both Volterra operators with a smooth kernel on $D$, and thus satisfy the conditions of \Cref{lemma:short_continuity}. Therefore
    $A_{Q,\ka} - (A_Q - A_{Q_\ka})$ belongs to $\cC([0,\infty))$ always and to $\cC^{m+1}([0,\infty))$ if $Q \in \cC^{m}([0,\infty))$ with $m\in \N_0$;  the analogous result for $A_{Q,\ka} -Q$ follows immediately (recall that $Q_\ka$, and hence $A_{Q_\ka}$, are smooth). The fact that $A_{Q,\ka}-Q$ can be extended continuously by zero to $[0,\infty)$ follows from estimate \eqref{eq:est_thm_1d_2} that we now prove.

    Using \eqref{eq:est_Gk} and \eqref{eq:est_Tk} gives
    \begin{equation}\label{eq:est_gronw}
        |\left(\cG_\ka^{-1}\cT_{\ka} -\Id\right)F (t)| \le |\ka| C_\ka \int_0^t |F(s)| \, \d s ,
    \end{equation}
    for some $C_\ka>0$.
    Applying this with $F =A_Q-A_{Q_\ka}$ yields the estimate
    \begin{equation} \label{eq:est_thm_1d_1}
            |A_{Q,\ka}(t) - (A_Q(t)-A_{Q_\ka}(t))| \le |\ka| C_\ka \int_0^t |A_Q(s)-A_{Q_\ka}(s)| \, \d s .
    \end{equation}
     Estimate \eqref{eq:est_thm_1d_2} follows adding and subtracting $Q_\ka-Q$ in the LHS of  \eqref{eq:est_thm_1d_1} and using  the estimate
    \begin{multline} \label{est:A_difference} 
       \left |A_{Q_1}(t) -Q_1(t) - \left( A_{Q_2}(t) -Q_2(t) \right) \right | \\
       \le (\alpha_{Q_1}(t) + \alpha_{Q_2}(t))  e^{t(\alpha_{Q_1}(t) + \alpha_{Q_2}(t))}  \int_{0}^t |Q_1(s)-Q_2(s)|\, \d s ,
    \end{multline}
    twice, which can be found in \cite[Section~4.1]{Radial_Born} (recall that $\alpha_Q$ was defined in \eqref{eq:alpha_q_def}).

    It remains to prove property (iv). If $Q$ is positive, the operator $H_Q = (-\partial_t^2   + Q(t) )$ has no negative eigenvalues nor a zero resonance, as remarked in \cite{Simon_spectral_III}. Then it follows from \cite[Theorem 3]{Simon_spectral_III} that $A_Q \in L^1(\R_+)$. Also, we have $A_{Q_\ka}  \in L^1(\R_+)$ if $\ka < (j_{0,1})^2$ by \Cref{lemma:m_function}, since $z_\ka$ can be taken to be $0$. It follows that $A_{Q,\ka} \in  L^1(\R_+)$ by \eqref{eq:A_Q}, since $\cG_\ka^{-1} \cT_\ka $ is bounded on $L^1(\R_+)$ by Lemmas \ref{lemma:Change_basis} and \ref{lemma:Change_basis_3}. This finishes the proof of the Theorem.
\end{proof}

\begin{proof}[Proof of \Cref{thm:stability_A_function}]
    The first statement was proved by Simon in \cite[Theorem 1.5]{Simon_spectral_I} for the case $\ka =0$. The case $\ka \neq 0$ follows from this and \eqref{eq:A_Q} by the local injectivity of the  operators $\cG_\ka^{-1}\cT_{\ka} $, which satisfy that $\cG_\ka^{-1}\cT_{\ka}F|_{(0,a)}$ vanishes iff $F|_{(0,a)}$ vanishes due to $\cG_\ka$, $\cG_\ka^{-1}$, $\cT_{\ka}$, $\cT_{\ka}^{-1}$ being Volterra operators (see \Cref{app:jost}).
    
    We now prove the stability estimate.
     Under the assumptions for $Q_j$ $j =1,2$ in the statement, by \cite[Theorem 5.1]{Radial_Born} we have that 
     \begin{equation}\label{eq:stability_au}
        \int_0^{a}  \left| Q_1(t)-Q_2(t) \right| \,\d t    < C(a,M)
        \left(\int_0^a  \left| A_{Q_1}(t) - A_{Q_2}(t) \right| \, \d t\right)^{1/(p'+1)},
    \end{equation}
    provided that
    \begin{equation*} 
        \int_0^a  \left| A_1(t) - A_2(t) \right| \, \d t < \min(1,a)^{(1+p')/p'}.
    \end{equation*}
    By \eqref{eq:A_Q} we have that $  A_{Q_1}- A_{Q_2} =  \cT_\ka^{-1} \cG_\ka \left ( A_{Q_1,\ka} - A_{Q_2,\ka}\right)$, and therefore it follows that
    \begin{equation*}
        \int_0^a  \left| A_{Q_1}(t) - A_{Q_2}(t) \right| \, \d t \le  C(a,\ka) \int_0^a  \left| A_{Q_1,\ka}(t) - A_{Q_2,\ka}(t) \right| \, \d t,
    \end{equation*}
    by \Cref{lemma:Change_basis_3} and the boundedness of $\cG_\ka$ in $L^1(\R_+)$ (see \Cref{lemma:Change_basis}).
    Inserting this in \eqref{eq:stability_au} finishes the proof of the theorem.
\end{proof}


\appendix

\section{Jost solutions and intertwining operators}\label{app:jost}

In this Appendix, we gather some properties of Jost solutions and related operators that are used in the proofs of \Cref{thm:A-amplitude} and \Cref{thm:stability_A_function}.
Given $\ka\in\IR$ set $Q_\ka(x):=-\ka e^{-2x}$. For every $z\in \IC_+$, the problem 
\begin{equation}    \label{eq:uz_def_b}
(-\partial_x^2 + Q_\ka    ) u_z = -z^2 u_z, \qquad \text{ on }   \R_+,
\end{equation}
possesses a unique solution $\psi_z(\cdot;\ka)\in L^2(\IR_+)$ satisfying the asymptotics:
\begin{equation} \label{eq:jost_asymptotic}
    \psi_z(x;\ka) = e^{-zx} (1+ o(1)) \qquad \text{as } \, x\to \infty.
\end{equation}
These solutions are called the \textit{Jost solutions} of equation \eqref{eq:uz_def_b}, 
and are known to satisfy the identity
\begin{equation} \label{eq:id_jost_sol} 
    \psi_z(x;\ka) = e^{-zx} + \int_{x}^\infty K_\ka (x,t) e^{-zt} \, \d t,
\end{equation}
where 
\begin{equation*}
    K_\ka\in\cC^\infty(\{(x,t)\in\IR^2 : 0\leq x\leq t\})
\end{equation*}
is the \textit{Gelfand-Levitan} kernel associated to the smooth potential $Q_\ka$, see for instance \cite[Chapter~3]{Marchenko_87}. By \cite[Lemma~3.3.1]{Marchenko_87}, this kernel satisfies the estimate
\begin{equation}\label{eq:marchenko_estimate}
|K_\ka(x,t)|\leq \frac{|\ka|e^{\frac{|\ka|}{4}}}{4}e^{-(x+t)} ,\qquad \text{for all } 0\leq x\leq t.  
\end{equation}

Set $D := \{(x,t)\in\IR^2 : 0\leq t \leq x\}$, let $G_\ka\in\cC^\infty(D)$ be given by
\begin{equation} \label{eq:G_kernel_explicit_2}
        G_\ka(x,t)  :=      4K_\ka(t,2x-t) 
       +    2\int_{t}^{2x-t}   K_\ka(t,s)K_\ka(t,2x-s) \,\d s, 
\end{equation}
and define, for $F\in L^1_{\loc}(\R_+)$:
\begin{equation} \label{eq:G_kernel_explicit_1}
    \cG_\ka  F (x) :=  F(x) + \int_0^x  G_{\ka}(x,t) F(t)\, \d t.
\end{equation}
This is a Volterra integral operator. In the next result, we gather  some simple properties of this class of operators. Recall that, for $s\in\IR$, $L^1_s(\R_+)$ stands for the $F\in L^1_\loc(\IR_+)$ such that 
\begin{equation*}
    \norm{F}_{L^1_s(\IR_+)}:=\int_0^\infty|F(x)|e^{-2sx}\,\d x <\infty.
\end{equation*}    
\begin{lemma} \label{lemma:short_continuity}
    Let $H\in\cC^\infty(D)$; the \textit{Volterra operator} of kernel $H$ is defined as
    \begin{equation*}
        \cT_H F(x) := \int_0^x H(x,t) F(t) \, \d t,\qquad F\in L^1_{\loc}(\R_+).
    \end{equation*}
    For every $x\in \R_+$, it satisfies
    \begin{equation*}
        \|\cT_H F\|_{L^1((0,x))} \leq \left(\int_0^x\sup_{t\in(0,y)} |H(y,t)|\,\d y \right)\|F\|_{L^1((0,x))}.  
    \end{equation*}
    Moreover, $\cT_H$ maps $L^1_\loc(\R_+)$ into $\cC([0,\infty))$, and $\cC^{m}([0,\infty))$ into  $\cC^{m+1}([0,\infty))$ for every $m\in \N_0$. If, in addition, there exist $C_H>0$ and $g\in L^1_{1/2}(\IR_+)$ such that
    \[
    |H(x,t)| \le  C_H e^{-(x-t)}|g(x-t)|,\qquad \text{ for all }(x,t)\in D;
    \]
    then $\cT_H$ maps $L^1_{z_0}(\R_+)$ for all $z_0 \geq 0$ into itself, and
    \begin{equation*}
        \|\cT_H F\|_{L^1_{z_0}(\R_+)} \leq C_H\|g\|_{L^1_{1/2}(\IR_+)} \| F \|_{L^1_{z_0}(\R_+)}, \quad \text{ for all }F\in L^1_{z_0}(\R_+).
    \end{equation*}
\end{lemma}
Note that $\cG_{\ka}: L^1_{\loc}(\R_+) \To L^1_{\loc}(\R_+)$ 
is bounded by the preceding result, since the kernel $G_\ka$ is smooth on $D$. In fact, $\cG_{\ka}$ enjoys more precise mapping properties. 

First, let $z_0\ge 0$ and define the linear functional
\begin{equation} \label{eq:SF_jost_def}
    \cJ_\ka(F)(z) :=  \int_0^\infty F(x) \psi_z(x;\ka)^2 \, \d x, \qquad F \in L^1_{z_0}(\R_+). 
\end{equation}
This integral is well defined for all $\ka \in \R$ and $z\ge z_0$ by \eqref{eq:jost_asymptotic}. 
\begin{lemma} \label{lemma:Change_basis}
    Let $\ka\in\IR$. 
    \begin{enumerate}[i)]
        \item The kernel $G_{\ka}$ satisfies the estimate
        \begin{equation} \label{eq:est:kernel}
            \sup_{0<t<x}  |G_{\ka}(x,t)| \le   |\ka|C_\ka e^{-x}, \quad\text{ for all }x\in\IR_+,
        \end{equation}
     where $C_\ka = e^{\frac{|\ka|}{4}}(1+|\ka|\frac{e^{\frac{|\ka|}{4}}}{4})$.
        \item For every $z_0\geq 0$, 
        \begin{equation*}
             \cG_\ka : L^1_{z_0}(\R_+)\To L^1_{z_0}(\R_+)
        \end{equation*}
        is a bounded isomorphism.
        \item $\cG_\ka^{-1}-\Id$ is a Volterra operator whose integral kernel $\tilde{G}_\ka$ is smooth on $D$ and satisfies 
        \begin{equation*} 
            \sup_{0<t<x}|\tilde{G}_\kappa(x,t)|\leq |\ka|C_\ka e^{|\ka|C_\ka} e^{-x}, \quad\text{ for all }x\in\IR_+.
        \end{equation*}
        In particular, for every $F\in L^1_{\loc}(\R_+)$ and $x\in \IR_+$,
        \begin{equation}\label{eq:est_Gk}
            |\cG_\ka^{-1}(F)(x)-F(x)|\leq  |\ka|C_\ka e^{|\ka|C_\ka} e^{-x}\int_0^x|F(t)|\,\d t.
        \end{equation}
        \item For every $z_0\geq 0$ and $F\in L^1_{z_0}(\R_+)$,  
        \begin{equation*}
        \cJ_\ka(F)(z) = \int_0^\infty \cG_{\ka} (F)(x) e^{-2zx} \, \d x, \qquad \text{for all }z\ge z_0.
        \end{equation*}
    \end{enumerate}
\end{lemma}
\begin{proof}
    To prove (i) use \eqref{eq:marchenko_estimate} to find
    \[
        4\sup_{0<t<x} |K_\ka(t,2x-t)|  \le 
        |\ka|e^{\frac{|\ka|}{4}}  e^{-2x}  ,
    \]
    and
    \begin{multline*}
         2\sup_{0<t<x} \int_{t}^{2x-t}   |K_\ka(t,s)K_\ka(t,2x-s)| \,\d s  \le  
        \frac{(|\ka|e^{\frac{|\ka|}{4}})^2}{4}  \sup_{0<t<x} (x-t)    e^{-2(t+x)}
        \le  c_\ka |\ka|e^{\frac{|\ka|}{4}}   e^{-x},
    \end{multline*}
    where the last estimate follows from the inequality $x \leq e^{x}$ for $x\geq0$, with $c_\ka:=\frac{|\ka|}{4}e^{\frac{|\ka|}{4}}$. This proves \eqref{eq:est:kernel} with $C_\ka = e^{|\ka|/4}(1+c_\ka)$. 
    
    The fact that $\cG_\ka$ is bounded in $L^1_{z_0}(\R_+)$ follows from (i) and \Cref{lemma:short_continuity}, and the existence of its inverse follows from the convergence of the Neumann series, which implies that:
    \begin{equation} \label{eq:Gk_inverse_def}
         \cG_\ka^{-1}=\sum_{n=0}^\infty (-1)^n (\cG_\ka -\Id)^{n}.
     \end{equation}
    In order to check that the Neumann series converges, note that $(\cG_\ka -\Id)^{n}$ is a Volterra operator whose integral kernel, 
    \begin{equation*}
        G_\ka^n(x,t)=\int_0^x \int_0^{x_1} \dots \int_0^{x_{n-1}} G_\ka(x,x_1)G_\ka(x_1,x_2) \dots G_\ka(x_{n-1},t)\, \d x_{n-1} \dots \d x_2 \d x_1,  
    \end{equation*} 
    satisfies 
    \begin{equation}\label{eq:est_kerneln}
    |G_\ka^n(x,t)|\leq \frac{C_\ka^{n-1} |\ka|^{n-1}}{(n-1)!} |\ka|C_\ka e^{-x}.
    \end{equation}
    To see this, use \eqref{eq:est:kernel} to estimate
    \begin{equation*}
        |G_\ka^n(x,t)|\leq C_\ka^n |\ka|^n e^{-x}
         \int_0^x e^{-x_1}  \dots \int_0^{x_{n-2}} e^{-x_{n-1}} \, \ \d x_{n-1} \dots \d x_1=C_\ka^n|\ka|^n e^{-x} T^n(1),
    \end{equation*} 
    where  $T(f) (x) := \int_0^x e^{-t} f(t) \, \d t$. One can show by induction that
    \begin{equation} \label{eq:T_n_formula}
        T^{n}(1)(x) = \frac{1}{n!}(1-e^{-x})^n \le \frac{1}{n!}.
    \end{equation}
    Therefore, $\cG_\ka^{-1}-\Id$ exists and it is a Volterra operator whose kernel $\tilde{G}_\kappa$ is smooth on $D$ and satisfies:
    \[
    |\tilde{G}_\kappa(x,t)|\leq \sum_{n=1}^{\infty}|G_\ka^n(x,t)|\leq |\ka|C_\ka e^{|\ka|C_\ka} e^{-x}.
    \]
    This completes the proofs of (ii) and (iii). Finally, (iv) follows from \eqref{eq:id_jost_sol}  and \eqref{eq:SF_jost_def} by straightforward computation.
\end{proof}

Recall the definition of the Laplace transform:
\begin{equation*}
    \cL(f)(z) : = \int_0^\infty f(x) e^{-zx} \, \d x.
\end{equation*}
For $\ka\in\IR$, define $ r_\ka(t) : = 2K_\ka (0,2t)$; by \eqref{eq:id_jost_sol} we have that 
\begin{equation*}
   \psi_z(0;\ka) = 1+ \cL (r_\ka)(2z).
\end{equation*}
Given $f,g \in L^1(\R_+)$ we denote
\begin{equation*}
   f \star g (x) = \int_{0}^x f(x-t)g(t) \, \d t.
\end{equation*}
Recall that $\cL g(z) \cL f(z) = \cL (f \star g) (z) $.

Let $\tau_\ka := 2 r_\ka + r_\ka\star r_\ka$ and define the Volterra operators
\begin{equation} \label{eq:Tk_2}
    \cT_\ka g(x) = g(x) + \int_0^x  \tau_\ka(x-t) g(t) \, \d t.
\end{equation}
\begin{lemma}\label{lemma:Change_basis_3}
    Let $\kappa\in\IR$.
    \begin{enumerate}[i)]
        \item For every $g\in L^1_\loc(\IR_+)$ it holds that
        \begin{equation} \label{eq:est_Tk}
        |\cT_\ka g(x) -g(x) | 
        \le C_\ka |\ka| \int_0^x  |g(t)| \, \d t.
        \end{equation}
        \item For every $z_0\geq 0$, $(\cT_\ka-\Id):L^1_{z_0}(\IR_+)\To L^1_{z_0}(\IR_+)$ is bounded with norm bounded by $|\ka|C_\ka$.
        \item The inverse operator $\cT_{\ka}^{-1}$ exists and satisfies, for every $x\in\IR_+$ and $F\in L^1_{\loc}(\R_+)$:
        \begin{equation}\label{eq:normtkmo}
        \norm{\cT_\ka^{-1} F}_{L^1((0,x))} \le   e^{x C_\ka} \norm{F}_{L^1((0,x))}.
        \end{equation}
        \item Let $z_0\geq0$. For every $g\in L^1_{z_0}(\IR_+)$ it holds that
        \begin{equation}\label{eq:Tk_1}
            \cL (\cT_\ka g)(2z) =  \psi_z(0,\ka)^2 \cL g(2z),\qquad\text{ for all }z\geq z_0 .
        \end{equation}
    \end{enumerate}
\end{lemma}
\begin{proof}
    The function
    $r_\ka (x)$ is smooth on $\R_+$ and by \eqref{eq:SF_jost_def},  
    \begin{equation} \label{eq:r_estimate}
        |r_\ka (x)| \le |\ka| \frac{e^{\frac{|\ka|}{4}}}{2}  e^{-2x},
    \end{equation}
    and
    \[
    |r_\ka\star r_\ka (x)|  \le \int_0^{x} |r_\ka(x -s)||r_\ka(s)| \, \d s \\
    \le  \ka^2 \frac{e^{\frac{|\ka|}{2}}}{4} \int_0^{x} e^{-2x} \, \d s = \ka^2 \frac{e^{\frac{|\ka|}{2}}}{4} xe^{-2x}.
    \]
    These estimates imply that
    \begin{equation} \label{eq:tau_est}
        \tau_\ka(x) \le |\ka| C_\ka (x+1) e^{-2x}.
    \end{equation}
    Estimate \eqref{eq:est_Tk} is a direct consequence of this. Property (ii) follows from \eqref{eq:tau_est} and \Cref{lemma:short_continuity}. To prove (iii) denote $\cM_\ka g= \tau_k \star g$, we show that the Neumann series
    \begin{equation} \label{eq:neumann}
       \sum_{n=0}^\infty (-1)^n\cM_\ka^{n}
    \end{equation}
    converges in operator norm. For $n\ge 1$, the integral kernel of $\cM_\ka^{n}$ is given by
    \begin{equation*}
        m_\ka^n(x,t):= \int_0^{x_1} \tau_\ka(x-x_1)  \dots \int_0^{x_{n-1}}  \tau_\ka(x_{n-1}-t) \, \d x_{n-1} \dots    \d x_1.
    \end{equation*}
    Since by \eqref{eq:est_Tk} $|\tau_\ka(x)| \le C_\ka |\ka|$ for all $x\in \R_+$, one has the estimate:
    \[
    |m_\ka^n(x,t)|\le (|\ka| C_\ka)^n \frac{x^{n}}{n!},\qquad \text{for all } (t,x)\in D.
    \]
    This shows that the series \eqref{eq:neumann} converges in operator norm in $L^1((0,x))$ for every $x\in\IR_+$, as claimed, and that \eqref{eq:normtkmo} holds. Finally, property (iv) is a straightforward computation:
    \begin{equation*} 
        \psi_z(0,\ka)^2 \cL g(2z)  = \cL g(2z) + 2 \cL (g \star r_\ka)(2z) +  \cL (g \star r_\ka \star r_\ka)(2z) \\ 
     = \cL (\cT_\ka g)(2z).\qedhere
    \end{equation*} 
\end{proof}

\section{Spherical harmonics and projectors}\label{app:sh}
In this appendix, we present some explicit computations involving spherical harmonics. Recall that a $d$-dimensional spherical harmonic of degree $\ell\in\IN_0$ is the restriction to the sphere $\IS^{d-1}\subset \R^d$ of a complex homogeneous polynomial $P$ in $d$ variables of degree $\ell$ that is harmonic, \textit{i.e.} 
\[
\Delta P(x)=0,\qquad x\in\IR^d.
\]
These functions form a vector space that we denote by $\H_\ell^d$, of dimension
\begin{equation*}
    N_{\ell,d}:=\dim \H_\ell^d=\binom{\ell+d-1}{\ell} - \binom{\ell+d-3}{\ell-2}=\frac{(\ell+d-2)! + \ell(\ell+d-3)!}{\ell! (d-2)!}.   
\end{equation*}
Spherical harmonics of different degrees are orthogonal in $L^2(\IS^{d-1})$. In fact,
\[
L^2(\IS^{d-1})=\bigoplus_{\ell\in\IN_0}\H_\ell^d,
\]
and, moreover, spherical harmonics diagonalize the Laplacian on the sphere:
\[
-\Delta_{\IS^{d-1}}|_{\H_\ell^d}=\ell(\ell+d-2)\Id_{\H_\ell^d},\qquad \forall\ell\in\IN_0.
\]
Important examples of spherical harmonics are the restrictions to $x\in\IS^{d-1}$ of the functions:
\[
(x\cdot\zeta)^\ell, \qquad \text{provided }\;\zeta\in\IC^d,\;\zeta\cdot\zeta=0.
\]
This follows from:
\begin{equation*}
    \Delta(x\cdot\zeta)^\ell=\ell(\ell-1)(x\cdot\zeta)^{\ell-2}(\zeta\cdot\zeta)=0,\qquad x\in\IR^d.
\end{equation*}
The orthogonal projector  $\mathcal{P}_{\ell,d}:L^2(\Sp^{d-1})\longrightarrow \H_\ell^d$  is given by
\begin{equation*}
\mathcal{P}_{\ell,d}f(x)= \frac{N_{\ell,d}}{\abs{\Sp^{d-1}}}\int_{\Sp^{d-1}} P_{\ell,d}(x\cdot y)f(y)\d{y},\qquad f\in L^2(\Sp^{d-1}),
\end{equation*}
where $P_{\ell,d}$ are the generalized Legendre polynomials
\begin{equation*}
    P_{\ell,d}(t)=(-1)^\ell R_{\ell}(1-t^2)^{-(\nu_d-\frac{1}{2})}\left(\dfrac{\d{}}{\d{t}}\right)^\ell(1-t^2)^{\ell+\nu_d-\frac{1}{2}},
\end{equation*}
and
\[
\nu_d:=\frac{d-2}{2},\quad R_{\ell}:=\frac{\Gamma\left(\frac{d-1}{2}\right)}{2^\ell\Gamma\left(\ell+\frac{d-1}{2}\right)}, \quad|\Sp^{d-1}|=\frac{2\pi^{d/2}}{\Gamma(d/2)}.
\]
In particular, the function 
\begin{equation}\label{e:zonal}
    Z_\ell(x\cdot y):=\frac{N_{\ell,d}}{|\IS^{d-1}|} P_{\ell,d}(x\cdot y)
\end{equation}
is a reproducing kernel for $\H_\ell^d$. 
The normalization in the definition of $P_{\ell,d}$ ensures that $P_{\ell,d}(1)=1$, and (\cite[Equation 2.67]{Atkinson2012})
\begin{equation}\label{e:normp}
    \frac{|\IS^{d-2}|}{|\IS^{d-1}|}\int_{-1}^{1}P_{m,d}(t)P_{n,d}(t)(1-t^2)^{\nu_d-\frac{1}{2}}\d t=\frac{1}{N_{m,d}}\delta_{m,n}.  
\end{equation}

It also follows from the definition that
\begin{equation}\label{e:parp}
    P_{\ell,d}(-t)=(-1)^\ell P_{\ell,d}(t).
\end{equation}
A spherical harmonic $Y_\ell\in\H_\ell^d$ is invariant by all rotations that fix $\xi\in\IS^{d-1}$ if and only if it is a multiple of $P_{\ell,d}(x\cdot\xi)$.

Recall that, given any $\zeta\in\IC^d$ such that $\zeta\cdot\zeta\neq 0$, we write:
\[
\widehat{\zeta}:=\frac{1}{\sqrt{\zeta\cdot\zeta}}\zeta.
\]
\begin{lemma}\label{prop:Analytic Extension} 
The following statements hold.
\begin{enumerate}[i)]
    \item Let $f:\C\to\C$ be holomorphic and let $x,z\in\C^{d}$ satisfy $x\cdot x\neq 0\neq z\cdot z$, then
    \begin{align}
    \int_{\Sp^{d-1}} P_{\ell,d}\left(\widehat{x}\cdot y\right)f\left(y\cdot \widehat{z}\right)\d{y}&=\abs{\Sp^{d-2}} P_{\ell,d}\left(\widehat{x}\cdot \widehat{z}\right)\int_{-1}^1 P_{\ell,d}(t) f(t)(1-t^2)^{\nu_d-\frac{1}{2}}\d{t}\nonumber\\ 
    &=\abs{\Sp^{d-2}}R_{\ell}  P_{\ell,d}\left(\widehat{x}\cdot \widehat{z}\right)\int_{-1}^1 \p_t^\ell f(t)(1-t^2)^{\ell+\nu_d-\frac{1}{2}}\d{t}.\label{eq:Analytic Extension 1}
    \end{align}
    \item Let $f_\ell\in\H^d_\ell$ and $z\in\C^d$, then
    \begin{equation}\label{eq:Analytic Extension 2}
        \int_{\Sp^{d-1}} f_\ell(y)\left(y\cdot z\right)^\ell\d{y}=\frac{2^{1-\ell}\pi^{d/2}\ell!}{\Gamma(\ell+\frac{d}{2})} f^{\IC}_\ell(z),
    \end{equation}
    where $f^{\IC}_\ell$ stands for the analytic extension of $f_\ell$ to $\IC^d$.
\end{enumerate}
\end{lemma}
\begin{proof}
\begin{enumerate}[i)]
\item We start by noting that the second equality follows directly from the definition of $P_{\ell,d}$ and integration by parts $l$ times. For $x,z\in\R^d\setminus\{0\}$ the first equality is known as the Funk–Hecke Formula \cite[Theorem 2.22]{Atkinson2012}, since both sides are complex analytic on $x,z$ whenever $x\cdot x\notin(-\infty,0]$ and $z\cdot z\notin(-\infty,0]$, then they are equal on the same set. Continuity extends the equality to $x\cdot x\neq 0\neq z\cdot z$.
\item For $z\in\Sp^{d-1}$ this is a special case of the Funk–Hecke Formula \cite[Equation 2.65]{Atkinson2012}. Homogeneity of degree $\ell$ of both sides extends the equality to $z\in\R^d$, and analyticity to $z\in\C^d$.\qedhere
\end{enumerate}
\end{proof}
The next result gives an explicit expression of the orthogonal projection onto spherical harmonics of the function
\begin{equation*}
    e_\zeta(x):=e^{\zeta\cdot x},\qquad \zeta\in\IC^d.
\end{equation*}
\begin{lemma}\label{prop:Proyection Radial}
Let $\zeta \in \C^d$ and $\ell\in\N_0$, then
\begin{equation*}
\mathcal{P}_{\ell,d} e_\zeta(x)=
\begin{cases}
\displaystyle\Gamma(d/2)N_{\ell,d}\left(\frac{2}{\sqrt{\zeta\cdot\zeta}}\right)^{\nu_d}I_{\ell+\nu_d}\left(\sqrt{\zeta\cdot\zeta}\right) P_{\ell,d}\left(x\cdot\widehat{\zeta}\right),&\zeta\cdot\zeta\neq 0,\\
\\
\displaystyle\frac{(x\cdot\zeta)^\ell}{\ell!},& \zeta\cdot\zeta= 0.
\end{cases}
\end{equation*}
\end{lemma}

\begin{proof}
Assume $\zeta\cdot\zeta\neq0$ and set $\displaystyle g(t)=e^{\sqrt{\zeta\cdot\zeta}\,t}$, then by \eqref{eq:Analytic Extension 1} we have
\begin{align*}
\mathcal{P}_{\ell,d}e_\zeta(x)&=\frac{N_{\ell,d}}{\abs{\Sp^{d-1}}}\int_{\Sp^{d-1}} P_{\ell,d}(x\cdot y)e_\zeta(y)\d{y}=\frac{N_{\ell,d}}{\abs{\Sp^{d-1}}}\int_{\Sp^{d-1}}P_{\ell,d}(x\cdot y)g\left(y\cdot \widehat{\zeta}\right)\d{y}\\
&=\frac{\abs{\Sp^{d-2}}N_{\ell,d}R_{\ell}}{\abs{\Sp^{d-1}}} P_{\ell,d}\left(x\cdot \widehat{\zeta}\right)\int_{-1}^1 g^{(\ell)}(t)(1-t^2)^{\ell+\nu_d-\frac{1}{2}}\d{t}\\
&= \frac{\abs{\Sp^{d-2}}N_{\ell,d}R_{\ell}}{\abs{\Sp^{d-1}}} P_{\ell,d}\left(x\cdot \widehat{\zeta}\right)\left(\sqrt{\zeta\cdot\zeta}\right)^\ell\int_{-1}^1 e^{\sqrt{\zeta\cdot\zeta}\,t}(1-t^2)^{\ell+\nu_d-\frac{1}{2}}\d{t}.
\end{align*}
Using now the integral representation for the modified Bessel functions:
\[
I_\nu(z)=\frac{1}{\sqrt{\pi}\Gamma\left(\nu+\frac{1}{2}\right)}\left(\frac{z}{2}\right)^\nu\int_{-1}^1 e^{zt}(1-t^2)^{\nu-\frac{1}{2}}\d t,
\]
and
\[
\frac{|\IS^{d-2}|}{|\IS^{d-1}|}=\frac{\Gamma(d/2)}{\sqrt{\pi}\Gamma((d-1)/2)},
\]
we find that the above computation simplifies to
\begin{align*}
\mathcal{P}_{\ell,d}e_\zeta(x) = \Gamma(d/2)N_{\ell,d}\left(\frac{2}{\sqrt{\zeta\cdot\zeta}}\right)^{\nu_d}I_{\ell+\nu_d}\left(\sqrt{\zeta\cdot\zeta}\right) P_{\ell,d}\left(x\cdot \widehat{\zeta}\right).
\end{align*}

Now assume $\zeta\cdot\zeta= 0$. In this case, for every $n\in\IN_0$, the function $\IS^{d-1}\ni x\mapsto(x\cdot\zeta)^n\in\IC$ is an element of $\H_n^d$. The result now follows from the fact that
\[
e_\zeta(x)=\sum_{n=0}^\infty \frac{(x\cdot \zeta)^n}{n!}.\qedhere
\]
\end{proof}
These tools allow us to prove the following result.
\begin{lemma}\label{prop:mainradial}
Let $\zeta_1,\zeta_2\in\C^d$ such that $\zeta_1\cdot\zeta_1 =\zeta_2\cdot\zeta_2=-\kappa\in\IC$ 
and $r\in\R_+$. Then:
\begin{equation*}
    \hp{\cc{\mathcal{P}_{\ell,d}e_{r\zeta_1}}}{\mathcal{P}_{\ell,d}e_{r\zeta_2}}_{L^2(\Sp^{d-1})} = 
    \begin{cases}
       \displaystyle
        (2\pi)^d   \frac{J_{\ell+\nu_d}(\sqrt{\kappa} r)^2}{(\sqrt{\ka} r)^{2\nu_d}}Z_{\ell,d}\left(\frac{\zeta_1\cdot\zeta_2}{\kappa}\right)
        ,&\kappa\neq 0\\
        \displaystyle 2\pi^{d/2} \frac{r^{2\ell}}{\Gamma(\ell+\frac{d}{2})\ell!}\left(\frac{\zeta_1\cdot \zeta_2}{2}\right)^\ell,&\kappa=0.
   \end{cases}   
\end{equation*}   
\end{lemma}
\begin{proof}
Assume first that $\kappa\neq 0$. By \Cref{prop:Proyection Radial} it follows that
\begin{align*}
    \hp{\cc{\mathcal{P}_{\ell,d}e_{r\zeta_1}}}{\mathcal{P}_{\ell,d}e_{r\zeta_2}}&_{L^2(\Sp^{d-1})}\\
    =\Gamma(d/2)^2&N_{\ell,d}^2\left(\frac{2}{r\sqrt{-\kappa}}\right)^{2\nu_d}I_{\ell+\nu_d}\left(r\sqrt{-\kappa}\right)^2 \int_{\IS^{d-1}}P_{\ell,d}\left(x\cdot\widehat{\zeta_1}\right)P_{\ell,d}\left(x\cdot\widehat{\zeta_2}\right)\d x\\
    =\Gamma(d/2)^2&N_{\ell,d}^2\left(\frac{2}{r\sqrt{\kappa}}\right)^{2\nu_d}(-1)^\ell J_{\ell+\nu_d}\left(r\sqrt{\kappa}\right)^2 \int_{\IS^{d-1}}P_{\ell,d}\left(x\cdot\widehat{\zeta_1}\right)P_{\ell,d}\left(x\cdot\widehat{\zeta_2}\right)\d x,
\end{align*}
where we have used the identity $I_\nu(i\cdot)=i^\nu J_\nu$, and the fact that $\sqrt{-\kappa}=-\mathrm{sign}(\Im\kappa)i\sqrt{\kappa}$. Taking into account \eqref{eq:Analytic Extension 1}, \eqref{e:normp}, and \eqref{e:parp} we find that
\begin{align*}
   \int_{\Sp^{d-1}}P_{\ell,d}\left(y\cdot\widehat{\zeta_1}\right)P_{\ell,d}\left(y\cdot\widehat{\zeta_2}\right)\d{y}&=\abs{\Sp^{d-2}}P_{\ell,d}\left(\widehat{\zeta_1}\cdot\widehat{\zeta_2}\right)\int_{-1}^1 P_{\ell}(t)^2(1-t^2)^{\nu_d-\frac{1}{2}}\d{t}\\
   &=\frac{(-1)^\ell\abs{\Sp^{d-1}}}{N_{\ell,d}}P_{\ell,d}\left(\frac{\zeta_1\cdot\zeta_2}{\kappa}\right).
\end{align*}
Hence, we obtain
\begin{align*}
     \hp{\cc{\mathcal{P}_{\ell,d}e_{r\zeta_1}}}{\mathcal{P}_{\ell,d}e_{r\zeta_2}}_{L^2(\Sp^{d-1})}
    =&|\IS^{d-1}|\Gamma(d/2)^2 N_{\ell,d} \left(\frac{2}{r\sqrt{\kappa}}\right)^{2\nu_d}J_{\ell+\nu_d}\left(r\sqrt{\kappa}\right)^2 P_{\ell,d}\left(\frac{\zeta_1\cdot\zeta_2}{\kappa}\right)\\
    =&|\IS^{d-1}|^2\Gamma(d/2)^2 4^{\nu_d} \frac{J_{\ell+\nu_d}(\sqrt{\kappa} r)^2}{(\sqrt{\ka} r)^{2\nu_d}} Z_{\ell,d}\left(\frac{\zeta_1\cdot\zeta_2}{\kappa}\right)\\
    =&(2\pi)^d\frac{J_{\ell+\nu_d}(\sqrt{\kappa} r)^2}{(\sqrt{\ka} r)^{2\nu_d}} Z_{\ell,d}\left(\frac{\zeta_1\cdot\zeta_2}{\kappa}\right).
\end{align*}
For $\kappa=0$,
\begin{equation*}
    \hp{\cc{\mathcal{P}_{\ell,d}e_{r\zeta_1}}}{\mathcal{P}_{\ell,d}e_{r\zeta_2}}_{L^2(\Sp^{d-1})}=\frac{r^{2\ell}}{(\ell!)^2} \int_{\Sp^{d-1}}(y\cdot\zeta_1)^\ell(y\cdot \zeta_2)^\ell\d{y}.
\end{equation*}
Since $(y\cdot\zeta_1)^\ell\in \H_\ell^d$, we can use \eqref{eq:Analytic Extension 2} to obtain
\begin{align*}
       \int_{\Sp^{d-1}}(y\cdot\zeta_1)^\ell(y\cdot \zeta_2)^\ell\d{y}=\frac{2\pi^{d/2}\ell!}{\Gamma(\ell+\frac{d}{2})}\left(\frac{\zeta_1\cdot \zeta_2}{2}\right)^\ell,
\end{align*}
and the result follows.
\end{proof}

\bibliographystyle{myalpha}
\bibliography{Refs}

\end{document}